\documentclass[11pt]{article}
\usepackage{euscript}
\usepackage{a4wide}
\usepackage[T2A]{fontenc}
\usepackage[utf8]{inputenc}
\usepackage[english]{babel}
\usepackage{amsfonts}
\usepackage{amssymb, amsthm}
\usepackage{amsmath}
\usepackage{graphicx}
\usepackage{geometry}
\usepackage{tikz}
\usetikzlibrary{calc}
\usetikzlibrary{decorations.pathmorphing}
\usepackage{bbm}
\usepackage{amsmath}
\usepackage{amsthm}
\usepackage{dsfont}
\usepackage{listings}
\usepackage{tikz}
\usepackage[nottoc]{tocbibind}

\usepackage{setspace}
\usepackage{float}

\usepackage[T1]{fontenc}

\usepackage{paralist}

\usepackage[colorlinks,link color=blue, cite color=blue,pagebackref,pdftex]{hyperref}

\providecommand{\keywords}[1]
{
  \small	
  \textbf{\textit{Keywords --- }} #1
}
\providecommand{\MSC}[1]
{
  \small	
  \textbf{\textit{MSC:}} #1
}

\renewcommand{\epsilon}{\varepsilon}

\def\geq{\geqslant}
\def\ge{\geqslant}
\def\leq{\leqslant}
\def\le{\leqslant}

\def\kratno{\lower.5ex\hbox{$\,\vdots\,$}}

\def\q#1.{\smallbreak\noindent\hskip15pt{\bf#1.}\enspace\ignorespaces} 
\def\dotline{\smallskip\hbox to \hsize{\dotfill}\medskip}

\def\norm[#1]{\| #1 \|}
\newcommand{\Z}{\mathbb{Z}}
\newcommand{\N}{\mathbb{N}}
\newcommand{\R}{\mathbb{R}}

\newcommand{\slim}{\sum\limits}

\newcommand{\polar}{\circ}

\newcommand{\relint}{\mathrm{relint}}

\newcommand{\argmin}{\operatorname*{arg\; min}}

\theoremstyle{plain}
\newtheorem{thm}{Theorem}[section]
\newtheorem*{thm*}{Theorem} 
\newtheorem{lm}{Lemma}[section]
\newtheorem{st}{Statement}

\theoremstyle{definition}
\newtheorem*{defn}{Definition}

\newtheorem{cor}{Corollary}[section]

\newtheorem{lemma}{Lemma}[section]

\theoremstyle{remark}
\newtheorem*{rem}{Remark}

\newcommand{\aff}{\mathop{\mathrm{aff}}\nolimits}

\newcommand{\bu}{\mathbf{u}}
\newcommand{\F}{\mathcal{F}}
\newtheorem{conjecture}{Conjecture}[section]
\newcommand{\lin}{\mathop{\mathrm{lin}}\nolimits}
\newcommand{\conv}{\mathop{\mathrm{conv}}\nolimits}
\newcommand{\pos}{\mathop{\mathrm{pos}}\nolimits}
\newcommand{\bx}{\mathbf{x}}

\title{\textbf{Discrete intrinsic volumes }}
\author{Mariia Dospolova \thanks{The work was supported by the Foundation for the
Advancement of Theoretical Physics and Mathematics “BASIS”.}
\thanks{The work was supported by Ministry of Science and Higher Education of the Russian Federation, agreement №
075-15-2019-1620.}} 

\date{} 
\begin{document}

\maketitle
\date{} 

\begin{abstract}
For a convex lattice polytope  $P\subset \mathbb R^d$ of dimension $d$ with vertices in $\mathbb Z^d$, denote by $L(P)$ its discrete volume which is defined as the number of integer points inside $P$. The classical result due to Ehrhart says that for a positive integer $n$, the function $L(nP)$ is a polynomial in $n$ of degree $d$ whose leading coefficient is the  volume of $P$. In particular, $L(nP)$ approximates the volume of $nP$ for large $n$.

In convex geometry, one of the central notion which generalizes the volume is the intrinsic volumes. The main goal of this paper is to introduce their discrete counterparts. In particular, we show that for them the analogue of the Ehrhart result holds, where the volume is replaced by the intrinsic volume.

We also  introduce and study a notion of Grassmann valuation which generalizes both the discrete volume and the solid-angle valuation introduced by Reeve and Macdonald.

\end{abstract}

 \thispagestyle{empty}


\MSC{Primary, 52B20, 52A39; secondary, 52B05, 52C07, 60D05.}

\newpage
\tableofcontents


\newpage

\section{Introduction and main results}

Consider the integer lattice $\Z^d$. The main object of our interest is a \emph{convex lattice polytope}  $P\subset \R^d$, that is, a convex hull of a finite number of points from $\Z^d$. Denote by $\mathcal P(\Z^d)$ the set of all (including the empty set) convex lattice polytopes in $\R^d$.

By definition, $\aff P$, the affine hull of $P$, is the least affine subspace of $\R^d$ containing $P$, and $\dim P$, the dimension of $P$, is the dimension of this subspace. We say that $P$ is of full dimension if $\dim P=d$. Equivalently, $P$ has a non-empty interior or the $d$-dimensional Lebesgue measure of $P$ is positive.

Denote by $|P|$ the $k$-dimensional Lebesgue measure of $P$, where $k=\dim P$. Thus for $P\ne\emptyset$ we have $|P|>0$. In particular, if $P$ coincides with a point, then $|P|=1$. If $P=\emptyset$, then, by definition, $\dim P=0$ and $|P|=0$.

\subsection{Discrete volume}
For $n\in\N^1$ define the $n$-dilate of $P$ by
\begin{align*}
    nP:=\{n\bu:\bu\in P\}
\end{align*}
and for $P$ of full dimension define the \emph{discrete volume} of $P$ as the number of points from $\Z^d$ lying in $P$:
\begin{align}\label{1446}
    L(P):=|P\cap\Z^d|.
\end{align}
This term arose from the easy observation (see, e.g.,~\cite{pC13}) that if $\dim P=d$, then
\begin{align}\label{1305}
    \lim_{n\to\infty}\frac{L(nP)}{n^d}=|P|.
\end{align}
If $\dim P=k<d$, then~\eqref{1305} still holds with zero right-hand side. However, if we we replace the normalization term $n^d$ by $n^k$,  we get (see, e.g.,~\cite[Section~5.4]{BR15})
\begin{align}\label{1435}
    \lim_{n\to\infty}\frac{L(nP)}{n^k}=\frac{|P|}{\det(P)} ,
\end{align}
where $\det(P)$ denotes the determinant of the $k$-dimensional lattice $\Z^d\cap \aff P$. 

Thus the natural generalization of the discrete volume to the lattice polytopes in $\R^d$ of any dimension is
\begin{align*}
    L(P):=\det(P)\cdot|P\cap\Z^d|,
\end{align*}
which coincides with~\eqref{1446} when $\dim P = d$. With this notation,~\eqref{1305} takes the following general form:
\begin{align}\label{2133}
    \lim_{n\to\infty}\frac{L(nP)}{n^{\dim P}}=|P|.
\end{align}

One of our goals is to naturally generalize the definition of the discrete volume of $P$ to the definition of the \emph{$k$-th discrete intrinsic volume} and to obtain the generalization of~\eqref{2133} with the $k$-th intrinsic volume in the right-hand side. 

To this end, given a convex cone $C\subseteq\R^d$, denote by $\alpha (C)$ the solid angle of $C$ measured with respect to the linear hull of $C$ as the ambient space, so we always have $\alpha (C)>0$ (similar to $|\cdot|$, see Subsection~\ref{Convex cones} for details).
The usual (not discrete) $k$-th intrinsic volume of $P$ is defined (see, e.g.,~\cite[Relation~4.23]{rS14}) as
\begin{align}\label{2107}
    V_k(P)=\sum_{F\in\F_k(P)}\alpha(N_F(P))\cdot|F|,
\end{align}
where $N_F(P)$ is the normal cone of $P$ at face $F$ (see Subsection~\ref{polytopes} for the definition) and $\F_k(P)$ denotes the set of all $k$-dimensional faces (for short, $k$-faces) of $P$. Note that if $\dim P = k$, then $V_k(P)=|P|$ and $V_{k+1}(P)=\dots=V_d(P)=0$ (the latter holds if $k<d$).

The first na\"{\i}ve idea for the discrete intrinsic volume is to consider the following quantity:
\begin{align}\label{2314}
    L_k(P)=\sum_{F\in\F_k(P)}\alpha(N_F(P))\cdot \det(F) \cdot L(F).
\end{align}
Indeed, it easily follows from~\eqref{2107} and~\eqref{2133} that
\begin{align}\label{727}
    \lim_{n\to\infty}\frac{L_k(nP)}{n^k}=V_k(P).
\end{align}
However, $L_k(\cdot)$ defined in~\eqref{2314} does not possess one important property: it is not a \emph{valuation}. A real function $\varphi:\mathcal P(\Z^d)\to \R^1$ is called a valuation if $\varphi(\emptyset)=0$ and for all $P, Q\in\mathcal P(\Z^d)$,
\begin{align}\label{728}
    \varphi(P\cup Q)=\varphi(P)+\varphi(Q)-\varphi(P\cap Q)\quad \text{whenever}\quad P\cup Q\in\mathcal P(\Z^d).
\end{align}
We say that $\varphi$ is \textit{translation-invariant} with respect to $\Z^d$ if $\varphi(t + P ) = \varphi(P)$ for all $P \in \mathcal{P}(\mathbb{Z}^d) $ and $t \in \Z^d$.
All valuations that we  consider are translation-invariant.

Both $V_k(\cdot)$ and $L(\cdot)$ are valuations on $\mathcal P(\Z^d)$ which implies many of their nice properties, some of them being considered in Subsection~\ref{2329}. Moreover, the intrinsic volumes defined in~\eqref{2107} are valuations on the set of convex compacts in $\R^d$.

Thus it is essential to find a definition for the discrete intrinsic volume which satisfies both~\eqref{727} and~\eqref{728}.

As a possible clue to this problem, let us consider another measure of discrete volume, a so-called \emph{solid-angle valuation} introduced by Reeve~\cite{jR57},~\cite{jR59} and Macdonald~\cite{iM63},~\cite{iM71} (see also~\cite[Section~13.1]{BR15}). To this end, let us represent~\eqref{1446} as 
\begin{align}\label{903}
    L(P)=\sum_{v\in\Z^d} \mathbbm {1}[v\in P].
\end{align}
Now  let us define the solid-angle valuation of $P$ as a slight modification of~\eqref{903}:
\begin{equation*}
  A(P)= \begin{cases}
 \slim_{v\in\Z^d} \mathbbm {1}[v\in P]\cdot\alpha(T_v(P)) &\text{ if $\dim P =d$,} 
   \\
   0 & \text { otherwise},
 \end{cases}
\end{equation*}
where $T_v(P)$ is the tangent cone of $P$ at point $v$ (see Subsection~\ref{polytopes} for the definition). In other words, when we count point inside $P$, we leave the interior points untouched and  multiply the boundary points by the corresponding solid angles. Since for large $n$ the proportion of the boundary points of $nP$ is negligible, as in~\eqref{1305} 
we have 
\begin{align}\label{623}
     \lim_{n\to\infty}\frac{A(nP)}{n^d}=|P|.
\end{align}
On the other hand,  $A(\cdot)$ differs from $L(\cdot)$ in a substantial way: if $P,Q,P\cup Q\in\mathcal P(\Z^d)$ and we also have $\dim (P\cap Q)<d$, then
\begin{align*}
    A(P\cup Q)=A(P)+A(Q).
\end{align*}
Clearly, $L(\cdot)$ lacks this property. Essentially, this is a reason why defined above $L_k(\cdot)$ failed to be a valuation. On the contrary, the generalization of $A(\cdot)$ which we are going to introduce now will turn out to be a valuation.

For $P\in\mathcal P(\Z^d)$ of any dimension and $k=0,1,\dots,d$, let
\begin{align*}
    A_k(P):=\sum_{v\in P\cap \Z^d}\sum_{F\in\mathcal F_k(P)} \mathbbm {1}[v\in F] \det(F) 
    \alpha(N_F(P))\alpha(T_v(F)),
\end{align*}
where $\alpha(T_v(F))$ is the solid angle of the tangent cone of $F$ at point $v$, $\alpha(N_F(P))$ is the solid angle of the normal cone of $P$ at face $F$.   

In our first theorem we collect some properties of $A_k(\cdot)$ including the connection with $A(\cdot)$. 
\begin{thm}\label{2000}
For any $k\in\{0,1,\dots,d\}$, the $k$-th discrete intrinsic volume $A_k(\cdot)$ is a translation-invariant valuation on $\mathcal P(\Z^d)$.  Furthermore, 
for all $P\in\mathcal P(\Z^d)$ we have
\begin{enumerate}
    \item $A_0(P)= 1$.
   \item $A_d(P)=A(P)$.
    \item If $\dim P< k$, then $A_k(P)=0$.
    \item $\lim_{n\to\infty}\frac{A_k(nP)}{n^k}=V_k(P)$, \quad $k=0,1,\dots,d$.
\end{enumerate}
\end{thm}
The proof of Theorem 
\ref{2000} is given in Section~\ref{809}.
The fourth part of Theorem \ref{2000} immediately follows from the first part of Theorem \ref{1021} below.

In the next subsection, we will take a closer look at $L(nP), A(nP)$, and $A_k(nP)$ as functions of $n$.

\subsection{Ehrhart's and Macdonald's polynomials}\label{2329}
The fundamental results due to Ehrhart~\cite{eE67},~\cite{eE67b} and Macdonald~\cite{iM63},~\cite{iM71} state that if $P$ is a convex lattice polytope in $\R^d$ of full dimension, then for $n\in\Z_+$ the functions $L(nP)$ and $A(nP)$ are polynomials in $n$ of degree $d$:
\begin{align}\label{924}
    L(nP)=\sum_{i=0}^d L^{(i)}(P)n^i, \quad A(nP)=\sum_{i=0}^d A^{(i)}(P)n^i.
\end{align}
These polynomials, which we denote by 
\begin{align*}
    L_P(t), A_P(t), t\in\R^1,
\end{align*}
are called\emph{ Ehrhart's and Macdonald's polynomials}.

It readily follows from~\eqref{1305} and~\eqref{623} that
\begin{align}\label{1009}
    L^{(d)}(P) = A^{(d)}(P) = |P|,
\end{align}
and it is known that (see \cite[Corollary 3.15., Theorem 13.8.]{BR15}) 
\begin{align}\label{1010}
    L^{(0)}(P) = 1 \quad\text{and} \quad A^{(0)}(P) = 0.
\end{align}
Moreover, it is not hard to show that
\begin{align}\label{1011}
    L^{(d-1)}(P) = \frac{1}{2} \sum_{F\in\F_{d-1}(P)}\frac{|F|}{\det(F)}\quad\text{and} \quad A^{(d-1)}(P) =A^{(d-3)}(P) = \dots = 0,
\end{align}
see~\cite[Theorem~5.6]{BR15} and second part of~\eqref{651} below.

The important properties of Ehrhart's and Macdonald's polynomials are the \emph{reciprocity relations}:
\begin{align}\label{651}
    L_P(-t)=(-1)^{\dim P}L(t\cdot\relint (P)) \quad\text{and} \quad A_P(-t)=(-1)^{\dim P} A_P(t),
\end{align}
where $\relint (P)$ denotes the relative interior of $P$, the interior with respect to the affine hull of $P$.

Later McMullen~\cite{pM77} generalized~\eqref{924} to all translation-invariant valuations on $\mathcal P(\Z^d)$. Specifically, he showed that if $P\in\mathcal P(\Z^d)$ and $\varphi$ is a valuation on $\mathcal P(\Z^d)$, then for $n\in\Z_+$, the function $\varphi(nP)$ is a polynomial in $n$ of degree $r\leq d$:
\begin{align}\label{1621}
    \varphi(nP)=\sum_{k=0}^r \varphi^{(k)}(P)n^k,
\end{align}
where $\varphi^{(k)}(\cdot)$ is a homogeneous valuation on $\mathcal P(\Z^d)$ of degree $k$.

Therefore, since the $k$-th discrete intrinsic volume $A_k(\cdot)$ is a valuation on $\mathcal P(\Z^d)$ (see Theorem~\ref{2000}), $A_k(nP)$ is also a polynomial in $n\in\Z_+$ (of degree $k$, which will be seen later):
\begin{align*}
    A_k(nP)=\sum_{i=0}^k A_{k}^{(i)}(P)n^i,
\end{align*}
where we tacitly assumed that $\dim P\geq k$, otherwise 
we would have $ A_k(nP)\equiv 0$.
We call this polynomial the \emph{$k$-th intrinsic Ehrhart polynomial} and denote it by $A_{k,P}(t), t\in\R^1$.

In our next theorem, we derive different properties of the intrinsic Ehrhart  polynomials  and its coefficients similar to~\eqref{1009},~\eqref{1010},~\eqref{1011}, and~\eqref{651}.
\begin{thm}\label{1021}
Let $P\in\mathcal P(\Z^d)$ and $k\in\{0,1,\dots,d\}$. If $\dim P<k$, then $A_{k,P}(t)\equiv0$. Otherwise, $A_{k,P}(t)$ is always even or odd 
polynomial of degree $k$,
\begin{align*}
    A_{k,P}(t) = \sum_{i=0}^k A_{k}^{(i)}(P)t^i
\end{align*}
with the following properties of the coefficients:
\begin{enumerate}
    \item $A_{k}^{(k)}(P)=V_k(P)$,
    \item $A_{k}^{(0)}(P)=0$.
\end{enumerate}
Moreover, $A_{k,P}(t)$ satisfies the following reciprocity law:
\begin{align*}
    A_{k, P}(-t) =  (-1)^{k}  A_{k, P}(t).
\end{align*}
\end{thm}
The proof of Theorem \ref{1021}
 will be given in Subsection~\ref{816}. Now we would like to introduce one more family of discrete measures which simultaneously generalizes \emph{both} the discrete volume $L(\cdot)$ and the solid-angle valuation $A(\cdot)$.

\subsection{Grassmann angle valuations}\label{1224}
If $C\subset\R^d$ is a convex cone with non-empty interior (i.e., $\dim C = d$) such that
\begin{align}\label{1145}
    C\ne\R^d,
\end{align}
then its solid angle $\alpha (C)$ can be calculated as the one-halfed probability to be non-trivially  intersected with the random line through the origin randomly chosen with respect to the Haar measure. This observation encouraged Gr\"unbaum~\cite{bG68} to introduce the following generalization of the solid angle: 
\begin{align*}
    \gamma_k(C):=\mathbb P[C\cap W_{d-k}\ne \{0\}],\quad k=0,1,\dots, d,
\end{align*}
where $W_{d-k}$ is a random $(d-k)$-dimensional linear subspace randomly chosen with respect to the Haar measure on the Grassmannian of all linear $(d-k)$-dimensional subspaces. Since by definition $\gamma_{d-1}(C)=2\alpha(C)$ (when $C\ne\R^d$) and $\gamma_0(C)\equiv 1$, the natural generalization of both $L(\cdot)$ and  $A(\cdot)$ is
\begin{align*}
      \tilde G_{k}(P) := \sum_{v \in P\cap\mathbb{Z}^d} \gamma_{k}(T_{v}(P)).
  \end{align*} 
However, there are 2 problems with this definition. The first one, for $k=d-1$ it is not with full accordance with $A(\cdot)$: if $C=\R^d$, then $\frac12\gamma_{d-1}(C)\ne\alpha(C)$. The second problem is more crucial: $G_{k}(P)$ is not a valuation for all $k\ne0,d$. It easily follows from the fact that $\gamma_k$ coincides for $(k+1)$-dimensional linear subspaces and half-subspaces. 

Let us suggest a possible solution. In the above definition of the solid angle we can omit Assumption~\eqref{1145} if instead of  a random line we consider the probability of the intersection with a random \emph{ray}, i.e. half-line. Then, a modification of Grassmann angles fully agreed with the solid angle is
\begin{align*}
    \alpha_k(C):=\mathbb P[C\cap W_{d-k}^+\ne \{0\}],\quad k=0,1,\dots, d,
\end{align*}
where $W_{d-k}^+$ is a random $(d-k)$-dimensional linear half-subspace randomly chosen with respect to the Haar measure on the Grassmannian of all linear $(d-k)$-dimensional subspaces (for details, see Subsection \ref{1716}). Let
\begin{align*}
      G_{k}(P) := \sum_{v \in P\cap\mathbb{Z}^d} \alpha_{k}(T_{v}(P)).
\end{align*} 
Now, with this modification of the Grassmann angles, $G_{k}(\cdot)$ becomes a valuation (see Theorem~\ref{1548}). We call it the \emph{$k$-th Grassmann angle valuation}. The basic properties of the Grassmann angle valuations are collected in the next theorem.
\begin{thm}\label{1548}
$G_{k}(\cdot)$ is a translation-invariant valuation on $\mathcal P(\Z^d)$. Moreover, for all $P\in\mathcal P(\Z^d)$  we have
\begin{enumerate}
    \item $G_{0}(P)=L(P)-1$.
    \item $G_{d-1}(P)=A(P)$. 
    \item $G_{d}(P)=0$.
    \item If $\dim P= k$, then 
    \begin{align*}
        0=G_{d}(P)=\dots=G_{k}(P)\leq G_{k-1}(P)\leq G_{k-2}(P) \leq \dots \leq G_{0}(P)=L(P)-1.
    \end{align*}
\end{enumerate}
\end{thm}
The proof is postponed to Subsection~\ref{821}.

Again, since the $k$-th Grassmann angle valuation $G_k(\cdot)$ is a valuation on $\mathcal P(\Z^d)$ it follows from the result of McMullen (see~\eqref{1621}) that $G_k(nP)$ is also a polynomial in $n\in\Z_+$:
\begin{align*}
    G_k(nP)=\sum_{i=0}^{d} G_{k}^{(i)}(P)n^i,
\end{align*}
where for simplicity we assume that $P$ is of full dimension.
We call this polynomial the \emph{$k$-th Grassmann polynomial} and denote it by $G_{k,P}(t), t\in\R^1$.

In the next theorem, we derive different properties of the Grassmann  polynomials  and its coefficients similar to~\eqref{1009} and ~\eqref{1010}.
\begin{thm}\label{824}
Let $P\in\mathcal P(\Z^d)$ be of full dimension and $k\in\{0,1,\dots,d - 1\}$. Then  $G_{k,P}(t)$ is a polynomial of degree $d$,
\begin{align*}
    G_{k,P}(t) = \sum_{i=0}^d G_{k}^{(i)}(P)t^i,
\end{align*}
with the following properties of the coefficients:
\begin{enumerate}
    \item $G_{k}^{(d)}(P)=|P|$,
    \item $G_{k}^{(0)}(P)=0$.
\end{enumerate}
\end{thm}
The proof of Theorem~\ref{824} will be given in Subsection~\ref{823}. In the next section we discuss the question of the positivity of the Grassmann polynomials coefficients.

\subsection{Reeve tetrahedron}\label{1013}
Although it follows from~\eqref{1009},~\eqref{1010},~\eqref{1011} that the $d$-th, $(d-1)$-th and the constant coefficients of Ehrhart's and Macdonald's polynomials are non-negative, in general it does not hold for all coefficients. It could be easily seen from the well-known example of Reeve~\cite{jR57} which was used by him  to show that higher-dimensional generalizations of Pick's theorem do not exist.

Specifically, for $h\in\N$ consider a tetrahedron  in $\mathbb{R}^3 $ defined as
\begin{align*}
    \Delta_{h} := \conv \{ (0,0,0) , (0,1,0) , (1,0,0) ,  (1,1,h) \},
\end{align*}
which is called the \emph{Reeve tetrahedron}. It was shown in~\cite[Section 3.7]{BR15} and~\cite{b09positivity} that
\begin{align*}
&L_{\Delta_{h}}(t) =\frac{h}{6}t^3 + t^{2} +\left(2 - \frac{h}{6} \right)t +1 \quad\text{and}\quad A_{\Delta_{h}}(t) =\frac{h}{6}t^3 +  \left ( S - \frac{h}{6} \right) t,
\end{align*}
where $ S = S (\Delta_{h})  $ is the sum of the solid angles at the
vertices of $\Delta_{h}$. Thus, for $h$ large enough, the linear coefficients in the Ehrhart and in the Macdonald polynomials are negative. Our next result shows that the same holds for all Grassmann polynomials of the Reeve tetrahedron.

In terms of Grassmann polynomials for $\Delta_{h}$ we can rewrite as
\begin{align*}
G_{0, \Delta_{h}}(t) &= L_{\Delta_{h}}(t) - 1 = \frac{h}{6}t^3 + t^{2} + \left( 2 - \frac{h}{6} \right) t, 
\\
G_{2, \Delta_{h}}(t) &= A_{\Delta_{h}}(t) = \frac{h}{6}t^3 + \left ( S - \frac{h}{6} \right) t.
\end{align*}
So it is enough to consider only $G_{1, \Delta_{h}}(t)$.
\begin{thm}\label{938}
The first Grassmann polynomial of the Reeve tetrahedron has the following form:
\begin{align*}
G_{1, \Delta_{h}}(t) = \frac{h}{6}t^3 + t^2 + \left( S - \frac{h}{6} \right) t. 
\end{align*}
\end{thm}
Note that it is well known that $S < \frac{1}{2} $ (see ~\cite[Proposition ~ 11]{b09positivity}). Therefore, $S - \frac{h}{6} < 0$  if $h \geq 3$. So, for the Reeve tetrahedron $\Delta_h$ with $h \geq 3$ the Grassmann polynomials have negative linear coefficients. Theorem~\ref{938} is proved in Subsection~\ref{943}.

\subsection{Positive valuations}\label{Positive valuations}
Let $P\in\mathcal P(\Z^d), \dim P = r$. As was observed in Section~\ref{1013}, it is not true that the coefficients of the Ehrhart polynomial $L_P(t)$ are always non-negative. However, in his groundbreaking works Stanley~\cite{stanley1980},~\cite{rS93} showed that if we represent $L_P(t)$ as
\begin{align*}
    L_P(t) \ = \ h^*_0(P) \binom{t+r}{r} + h^*_1(P) \binom{t+r-1}{r} + \cdots 
                 + h^*_r(P) \binom{t}{r},
\end{align*}
then we have $h_0^*(P),h_1^*(P),\dots,h_r^*(P)\geq 0$, and moreover, if $P'\in\mathcal P(\Z^d)$ such that $P'\subset P$, then
\begin{align*}
    0\leq h_k^*(P')\leq h_k^*(P),\quad k=0,1,\dots d,
\end{align*}
where we set $h_k^*(P)=0$ for $k > \dim P$.\\
Later the similar result was proved for the Macdonald polynomial in~\cite{b09positivity}.

Encouraged by these 2 examples, Jochemko and Sanyal~\cite{jochemko2018combinatorial} introduced the following notion. Let $\varphi$ be a valuation on $\mathcal P(\Z^d)$ and $P\in\mathcal P(\Z^d)$ with $\dim P =r$. Let us represent~\eqref{1621} in the following form:
\begin{equation*}
    \varphi(tP) \ = \  h^\varphi_0(P) \binom{t+r}{r} + h^\varphi_1(P)
    \binom{t+r-1}{r} + \cdots + h^\varphi_r(P) \binom{t}{r}.
\end{equation*}
Following the definitions and notation of~\cite{jochemko2018combinatorial}, we call a valuation $\varphi$ \textit{combinatorially positive} if
$h^\varphi_i(P) \ge 0$ and \textit{combinatorially monotone} if  $h^\varphi_i(P') \le
h^\varphi_i(P)$ whenever $P' \subseteq P$.

As was mention above, the discrete volume and the solid angle valuation are combinatorially positive and monotone. The natural question is if it is true for the discrete intrinsic volumes and the Grassmann angle valuations? As for the former, we conjecture that the answer is positive.
\begin{conjecture}
For any $k\in\{0,1,\dots,d\}$ the $k$-th discrete intrinsic volume is combinatorially positive and monotone.
\end{conjecture}

Concerning the Grassmann angle valuations, our next theorem shows that in general, the answer is negative.
\begin{thm}\label{1206}
The $k$-th Grassmann valuation  is not combinatorially positive for $0 \leq k \leq d-2$.
\end{thm}
The proof is given in Subsection~\ref{1206}. The main ingredient of the proof is the following complete characterization of combinatorially positive and combinatorially monotone valuations which was obtained in~\cite{jochemko2018combinatorial}. Assume that $\varphi$ is a translation-invariant valuation on $\mathcal P(\Z^d)$. Then, the following statements are equivalent:
\begin{enumerate}[\rm (i)]
        \item $\varphi$ is combinatorially monotone; 
        \item $\varphi$ is combinatorially positive; 
        \item For every simplex $\Delta \in  \mathcal{P} (\mathbb{Z}^d) $
        \[
            \varphi(\relint(\Delta)) \ := \ \sum_{F \in \mathcal{F}(\Delta)} (-1)^{\dim \Delta-\dim F}
            \varphi(F) \ \ge \ 0,
        \]
       
        \end{enumerate}
where the sum is taken over all faces $F$ of $\Delta$.

Note that from that we can conclude that due to Theorem~\ref{1206}  the $k$-th Grassmann valuation  is also not combinatorially \emph{monotone} for $0 \leq k \leq d-2$.

Let us conclude this introductory section by describing how the rest of the paper is organized. In the next section, we collect the necessary notion, definitions and  facts from the convex geometry. Some of them have been already briefly introduced in this section, however, for the subsequent 2 sections we will need a more detailed account of the theory. The detailed proofs of all theorems announced in this section are given in Section~\ref{1300}. Section~\ref{1301} lies a little apart from the main line of our work. To introduce the generalization of the discrete volume and the solid-angle valuation which still keep the valuation property, it was necessary to modify the original definition of the Grassmann angles. However, besides that, it turned out that these slightly modified Grassmann angles possess many interesting properties which, in some sense, make them preferable to the original ones. We consider this question in details in Section~\ref{1301}.


\section{Preliminaries}\label{1452}
\subsection{Convex sets}\label{2142}
For a set  $K \subset \mathbb{R}^d$  denote by $\conv K$ its \textit{convex hull},
\begin{align*}
\conv K := \big\{\sum_{i=0}^{k} \lambda_{i} x_{i} : x_1, \dots, x_k \in K, \lambda_{1}, \dots,\lambda_{k} \geq 0, \sum_{i=0}^{k} \lambda_{i} = 1,  k\in \mathbb{N} \big\},
\end{align*}
and by $\pos K$ -- its conic (or positive) hull:
$$\pos K :=  \big\{\sum_{i=0}^{k} \lambda_{i} x_{i} : x_1, \dots, x_k \in K, \lambda_{1}, \dots,\lambda_{k} \geq 0, k\in \mathbb{N} \big\} = \{ \lambda \bx : \bx \in \conv K, \lambda \geq 0 \}.  $$

Let now $K$ be a \emph{compact} convex subset of $\R^d$. 
 Then, the \textit{intrinsic volumes} $V_0(K), \ldots, V_d(K)$ are defined as the coefficients in the Steiner formula
\begin{equation*}
|K+rB^d|=\sum_{k=0}^d \kappa_{d-k} V_k(K) r^{d-k}, \quad r\geq 0,
\end{equation*}
where 
$B^k$ denotes the $k$-dimensional unit ball, and $\kappa_k:=|B^k|=\pi^{k/2}/\Gamma(\frac k 2 +1)$ is the volume of $B^k$ ($\kappa_0:=1$). In the other words, the volume of expansion is a polynomial whose coefficients depend on the set $K$. 

There is an equivalent way to define the intrinsic volumes by Kubota's formula \cite[Section 6.2]{SW08}:
\begin{equation*}
V_k(K)= \binom dk\frac{\kappa_d}{\kappa_k \kappa_{d-k}}\mathbb{E}|(K|W_k)|,
\end{equation*}
where $W_{k}$ is a random $k$-dimensional linear subspace of $\mathbb{R}^d$ uniformly chosen with respect to the Haar measure on the Grassmannian of all such subspaces, and $K|W_k$ denotes the orthogonal projection of $K$ onto $W_k$.

In particular, 
 $V_d(\cdot)$ is the $d$-dimensional volume,  $V_{d-1}(\cdot)$ is half the  surface area, and $V_1(\cdot)$ is the mean width, up to a constant factor.

The intrinsic volumes of a set have the property of being independent on the dimension. 
This means that if we embed $K$ into $\R^N$ with $N\geq d$, the intrinsic volumes will be the same.

\subsection{Polyhedral sets} \label{Polyhedral sets}
An intersection of finitely many closed half-spaces of the form
\begin{align*}
    \bx = \{(x_1,\dots, x_d)\in\R^d:a_1 x_1 + a_2 x_2 + \dots + a_d x_d \leq b \text{ for some } a_1,\dots,a_d,b\in\R^1\}
\end{align*}
is called a \emph{polyhedral set} in $\R^d$. In our paper, we  mostly deal with 2 special cases: a \emph{polyhedral cone} and a \emph{convex polytope} which will be the objects of the following 2 subsections. In this subsection, we introduce the basic notion and definitions applied to them both.

Let $P\subset\R^d$ be a polyhedral set.
A linear hyperplane (linear subspace of codimension one) $H$,  such that  
$P$ lies entirely
on one side of $H$, is called a $\textit{ supporting
hyperplane }$ of  $P$. 
A $\textit{face}$ of $P$ is either a set of the form $P \cap H$, where $H$ is a supporting hyperplane
of $P$,  or $P$ itself.\\
The dimension of a $P$ is defined as the dimension of its linear hull (the minimal linear subspace, containing $P$:
\begin{align*}
    \dim P:= \dim\lin P,
\end{align*}
and the relative interior of $P$ is defined as the interior of $P$ with respect to $\lin P$ and denoted by $\relint (P)$. The same definitions are also applied to the faces of $P$, since they are  polyhedral sets as well.

Let us denote by $\mathcal {F} (P)$ the set of all faces of convex cone $P$, including the empty set anf $P$ itself, also  $\mathcal {F}_k (P)$ denotes the set of $k$-dimensional faces, and $f_k (P)$ denotes the number of $k$-faces of $P$. It is easy to see that 
\begin{align*}
 P = \bigcup_{F\in \mathcal{F} (P)} \relint(F).
\end{align*}

\subsection{Polyhedral cones} \label{Convex cones}
A non-empty set $C \subset \mathbb{R}^d $ is called a $\textit{convex cone}$ or just a $\textit{cone} $ if $C$ is a convex set such that $\lambda C=C$ for all $\lambda>0$. A polyhedral set which is also a cone is called a \emph{polyhedral cone}. 

Specifically, linear subspaces are polyhedral, and polyhedral cones are closed. In the following, we will assume that all cones $C$ are polyhedral and non-empty 
unless
otherwise stated.
Following \cite{amelunxen} let us recall some basic facts about polyhedral cones.

A polyhedral cone is called \textit{pointed} if the origin $0$ is a zero-dimensional face,
or, which is the same, if it does not contain a linear
subspace of positive dimension.


For a polyhedral cone $C\subseteq\R^d$, the following relation obtained by Euler is well known (see, e.g., \cite{lawrence}):
\begin{equation}\label{Euler}
  \sum_{i=0}^d (-1)^i f_i(C) =
  \begin{cases}
     (-1)^{\dim C} & \text{ if } C \text{ is a linear subspace,}
  \\ 0 & \text{ otherwise.}
  \end{cases}
\end{equation}

The \emph{solid angle} of $C$ is defined as the probability for the random vector $U$ uniformly distributed over $\lin C\cap\mathbb S^{d-1}$ (the unit sphere in the span of $C$) to hit  $C$:
\begin{align*}
   \alpha(C):= \mathbb P[U\in C].
\end{align*}
We stress that $\alpha(C)$ is measured inside $\lin C$ and does not depend on the ambient space, so we always have $\alpha(C)>0$. By definition, $\alpha(\{0\})=1$.

The$ \textit { polar } $cone of the convex polyhedral cone $ C $ is the set 
\begin{equation*}
C ^ {\polar} := \{v \in \R^d  \;: \; \forall w \in C, \langle w, v \rangle \leq 0 \}. \end{equation*}
Note that $C^{\polar}$ is also a convex polyhedral cone.

Let us recall the basic properties of the polar cones:
\begin{enumerate} 
    \item 
    If $C$ is a linear subspace, then $C^{\polar}=C^{\perp}$ is the orthogonal complement;
    \item $C^{\polar \polar} :=\left(C^{\polar}\right)^{\polar}= C$;
    \item  If $C\subseteq D$, then $C^\polar\supseteq D^\polar$;
 \end{enumerate}
 For a polyhedral cone $C\subseteq\R^d$, denote by $\Pi_C$ the metric projection defined as
\begin{equation*}\label{eq:orthodecom}
 \Pi_C(x) := \argmin \{\norm[x-y]^2 \mid y \in C\} .
\end{equation*}
 The Moreau decomposition of a point $x \in \mathbb{R}^d $ is the sum representation
\begin{equation*}\label{eq:moreau}
x = \Pi_C(x)+\Pi_{C^{\circ}}(x),
\end{equation*}
where $\Pi_C(x)$ and $\Pi_{C^{\circ}}(x)$ are orthogonal.

\vspace{1cm}

The conic counterparts of the intrinsic volumes are the \emph{conic intrinsic volumes}. They are defined for an arbitrary convex cone, however for our purposes it is more convenient to use the following definition which is applied for a polyhedral cone only.

Let $C \subseteq \R^d$ be a polyhedral cone and let $U$ be a random vector uniformly distributed over the unit sphere $\mathbb{S}^{d-1}$. Then for $0\leq k\leq d$ we define the $k$-\textit{th conic intrinsic volume} as the probability that the metric projection of $U$ onto $C$ lies in relative interior of a $k$-dimensional face of $C$:
\begin{align*}
    \upsilon_k(C) = \mathbb{P}\{\Pi_C(U)\in \cup_{F\in\mathcal F_k(C)}\relint (F)\}.
\end{align*}
In particular, if $\dim C=k$, then by definition,
\begin{align*}
    \upsilon_k(C) = \alpha(C).
\end{align*}

It immediately follows from the definition that the conic intrinsic volumes form a probability distribution on~$\{0,1,\ldots,d\}$ for a fixed cone $C$:
\begin{align*}
    \sum\limits_{k=0}^{d}\upsilon_k(C) = 1.
\end{align*}
In particular, if $C$ is a linear subspace of dimension $j$, then $ \upsilon_j(C) = 1$ and $ \upsilon_k(C) = 0$ for $ k \neq j$.\\

The conic intrinsic volumes satisfy the following version of the Gauss--Bonnet theorem (see~\cite[Section~6.5]{SW08}):
\begin{align}\label{gauss-bonnet}
\sum_{k=0}^d(-1)^k \upsilon_k(C)
    = \begin{cases}
        (-1)^{\dim C} & \text{ if } C \text{ is a linear subspace,}
        \\
        0 & \text{ otherwise.}
      \end{cases}
\end{align}
For a polyhedral cone $C$ it was shown by Gr\"unbaum~\cite[Theorem~2.8]{bG68} that
\begin{align}\label{intrinsic1}
  (-1)^k \upsilon_k (C) = \sum_{F\in \mathcal{F} (C)} (-1)^{\dim F} \upsilon_k (F).
   \end{align}
Also, it is easy to see that the conic intrinsic volumes of the polar cone satisfy
\begin{align*}
    \upsilon_k(C^{\polar})= \upsilon_{d-k}(C).
\end{align*}
For $k > d$ we write by definition $\upsilon_k(C) = 0$.

\vspace{1cm}

Another important geometric characteristic of a convex cone  closely associated with the  conic intrinsic volumes is the \emph{ Grassmann angles} which have been  introduced and studied by  Gr\"unbaum~\cite{bG68}. 
 Define the $k$-th \emph{Grassmann angle} of the convex cone $C \subseteq \R^d $ 
 as the probability for $C$ to be intersected by the random $(d-k)$-plane $W_{d-k}$ (defined in Subsection~\ref{2142}) non-trivially:
\begin{equation*}\label{1138}
\gamma_k(C):=\mathbb{P}[C\cap W_{d-k}\ne\{0\}].
\end{equation*}
It is not hard to prove that for any convex cone $C \subseteq \R^d $ with $C \neq \{0\}$,
\begin{align} \label{inequality}
1 = \gamma_0(C) \geq  \gamma_1(C) \geq  \ldots \geq   \gamma_d(C) = 0.    
\end{align}
If $C$ is of full dimension, its {solid angle} can be expressed in terms of the Grassmann angles as follows:
\begin {align}\label{sol2}
  \alpha(C): = \frac {1} {2} \gamma_ {d-1} (C) + \frac {1} {2} \mathbbm {1} [C = \mathbb{R} ^ d]=  \frac {1} {2}\mathbb{P}[C\cap W_{1}\ne\{0\}]+ \frac {1} {2} \mathbbm {1} [C = \mathbb{R} ^ d].
\end {align}
In \cite{bG68}, Gr\"unbaum proved that, like the solid angle, the Grassmann angles do not depend on the dimension of the ambient space: if we embed $C$ in $\R^N$ with $N \geq d$, the Grassmann angles will be the same. Thus for a linear $k$-plane  $L_k\subset\R^d$, $k\in \{1,\ldots,d\}$, we have
\begin{equation*}\label{2256}
\gamma_0(L_k) = \ldots = \gamma_{k-1}(L_k) = 1,\quad \gamma_k(L_k) = \ldots = \gamma_d(L_k) = 0.
\end{equation*}
For $C=\{0\}$, we have $\gamma_0(C) = \gamma_1(C)= \ldots = 0$.

The mentioned above connection between the conic intrinsic volumes and the Grassmann angles is expressd via the Crofton formula (see, e.g.,~\cite[p.~261]{SW08}): for all $k=0,1,\dots,d$ we have
\begin{equation}\label{1818}
\gamma_k(C)=2(\upsilon_{k+1}(C)+\upsilon_{k+3}(C)+\dots)
\end{equation}
provided that the cone $C$ is not a linear subspace.

\subsection{Convex polytopes}\label{polytopes}
A \emph{bounded} polyhedral set in $\R^d$ is called a \emph{convex polytope}, or just a \emph{polytope}. An equivalent definition for the convex polytope is a convex hull of finitely many points in $\R^d$. The equivalence of these two definitions is proved in \cite[Appendix A]{BR15}.

For non-negative $n$ the $n$-th \textit{dilation} of a polytope $P$ is defined by: $nP= \{ nx : x \in P\}$.

Faces of dimension
$0, 1$ and dim$P-1$ are called \textit{vertices, edges}, and \textit{facets} respectively.
For polytopes, Euler's relation (cf.~\eqref{Euler}) takes the form (see, e.g., \cite[Theorem 5.2.]{BR15})
\begin{align}\label{Eulerp}
     \sum_{i=0}^d (-1)^i f_i(P)=1.
\end{align}


Further, for a convex polytope $P$, face $F$ of $P$ and $v \in \relint(F) $ we define a \textit{tangent cone} $T_v(P)$ as
\begin{align*}
    T_v(P) = \pos(P-v).
\end{align*}
It is easy to understand that $T_v(P)$ is a polyhedral cone and for any two different points $ v_ {1}, v_ {2} \in \relint (F) \text { we have } T_ {v_ {1}} (P) = T_ {v_ {2}} (P) $, hence the tangent cone of  $P$ at  point $v$ depends only on the face, in the relative interior of which $v$ lies.
Sometimes we will denote the tangent cone of $P$ corresponding to the face $F$ by $T_ {F} (P).$
Finally, the \textit{normal cone} $N_F(P)$ to $P$ at $F$ is defined by identity $N_F(P) = T_ {F} (P) ^{\polar}$. 

The solid angles of the tangent cones of a polytope $\alpha_{F,P} := \alpha(T_ {F} (P))$ (inner angles) are connected according to the Brianchon--Gram relation (see, e.g., \cite[Corollary 13.9.]{BR15}) which is a multi-dimensional generalization of the fact that the angles of a plane triangle sums up to $\pi/2$:
\begin{align}\label{Brianchon}
    \slim _{F\in \mathcal{F}(P) } \alpha_{F,P} (-1)^{dim F} = 0. 
\end{align}

If $P$ is a polytope, then it is known that (see~\cite[Relation~4.23]{rS14}) its $k$-th intrinsic volume can be calculated as
\begin{align}\label{2107a}
    V_k(P)=\sum_{F\in\mathcal{F}_k(P)}\alpha(N_F(P))\cdot | F| = \sum_{F\in\mathcal{F}_k(P)}\upsilon_{k} (T_F(P))\cdot | F|.
\end{align}

\section{Modified Grassmann angles}\label{1301}
In this section we introduce the modified definition of Grassmann angles and consider their main properties, as well as the relationship with the original definition.

\subsection{Definition and basic properties}\label{1716}
Let us recall that $W_k$ denotes a random $k$-dimensional linear subspace randomly chosen with respect to the Haar measure on the Grassmannian of all linear $k$-dimensional subspaces. Let $U$ be a random vector uniformly distributed over the unit sphere $\mathbb{S}^{d-1}$ independently from $W_k$. Denote by $W^+_k$ a random closed half-subspace defined as
\begin{align*}
    W^+_k:=W_k\cap U^\perp_+,
\end{align*}
where $U^\perp_+$ is the closed half-space containing $U$ with the boundary orthogonal to $U$:
\begin{align*}
    U^\perp_+:=\{\bx\in\R^d:\langle \bx,U\rangle\geq 0\}.
\end{align*}
Also denote by $ W_ {k} ^ {-}$ the complementary random half-subspace 
\begin{align*}
    W^-_k:=W_k\setminus\relint(W^+_k).
\end{align*}
We define the $ k $-th \textit {modified Grassmann angle} $ (k \in \{0, \dots, d \}) $ of a convex cone $ C  \subseteq \R^d$ as
   \begin{align*}
        \alpha_ {k} (C): = \mathbb {P} [W_ {d-k} ^ {+} \cap C \neq \{0 \}].
   \end{align*}
Let us recall that the original Grassmann angles are defined as
\begin{align*}
    \gamma_k(C):=\mathbb P[C\cap W_{d-k}\ne \{0\}],\quad k=0,1,\dots, d.
\end{align*}

The next theorem establishes a connection between these two definitions.
\begin{thm} \label{connection}
Let $C \subseteq \R^d$ be a convex cone such that $ C \neq \mathbb{R}^{d-k+1}$ and $0\leq k \leq d-1$. Then
\begin{align*}
 \alpha_ {k} (C) =  \frac{\gamma_k(C)+\gamma_{k+1}(C)}{2}.
\end{align*}
\end{thm}
Note that for $ k=1$  and $C \neq \mathbb{R}^d$ we really have
$$  \mathbb{P} [W_{1}^{+} \cap C \neq  \{0\} ] =  \frac{1}{2}  \mathbb{P} [W_{1} \cap C \neq  \{0\} ] + 0  = \frac{1}{2}  \mathbb{P} [W_{1} \cap C \neq  \{0\} ].$$
The main ingredient of the proof of Theorem~\ref{connection} is the following lemma.
\begin{lemma} \label{633}
Let $ C \subseteq \mathbb {R} ^ d $ be a convex cone, $ 1 \leq k \leq d $ and $ C \neq \mathbb {R} ^ {d-k + 1} $.  Then
\begin {align}\label{connection^}
\mathbb {P} \left [(W_ {k} ^ {+} \cap C \neq \{0 \}) \bigcap (W_ {k} ^ {-} \cap C \neq \{0 \}) \right] = \mathbb {P} \left [W_ {k-1} \cap C \neq \{0 \} \right].
\end{align}
\end{lemma}
The proofs of all results of this section are collected in Subsection~\ref{2247}. 

Now let us present a modified version of the Crofton formula, cf.~\eqref{1818}:
\begin{thm}[New version of Crofton formula]\label{Crofton}
Let $ C \subseteq \mathbb {R} ^ d $ be a convex cone; then
\begin{align}\label{749}
 \alpha_{k} ( C )  = \sum_{i\geq 1}  \upsilon_{k+i} (C).
\end{align}
\end{thm}
Let us stress that unlike in~\eqref{1818}, here we do not assume that $C$ is not a linear subspace, so~\eqref{749} holds for \emph{any} convex cone $C$.

The following is an immediate consequence of Theorem \ref{Crofton}.
\begin{cor}
With modified definition of the Grassmann angle, we have
\begin{enumerate}
    \item If $\dim C = k$, then $\alpha_{k-1} ( C ) = \upsilon_{k} (C) =  \alpha (C)$;
    \item $\alpha_{0} ( C ) = \slim_{i\geq 1}  \upsilon_{i} (C) =\slim_{i\geq 0}  \upsilon_{i} (C) -\upsilon_{0} (C) =  1 -\upsilon_{0} (C).$
\end{enumerate}
\end{cor}
Also, by Theorem \ref{Crofton}, $\upsilon_k(C) = \alpha_{k-1} (C) - \alpha_{k} (C) $ for $k\geq1$. \\
Thus, using \eqref{intrinsic1} and \eqref{gauss-bonnet}  we obtain the following formulae:
\begin{cor}
Let $C\subseteq\mathbb{R}^d$ be a polyhedral cone. Then for $k\geq 1$
\begin{align*}
  (-1)^k (\alpha_{k-1} (C) - \alpha_{k} (C)) = \sum_{F\in \mathcal{F} (C)} (-1)^{\dim F}  (\alpha_{k-1} (F) - \alpha_{k} (F)),
  \end{align*}
  where $\mathcal{F}(C)$ - the set of all faces of the cone $C$.
\end{cor}
\begin{cor}
For a polyhedral cone $C$:
\begin{align*}
1 - \alpha_0(C) +\sum_{k=1}^d(-1)^k (\alpha_{k-1} (C) - \alpha_{k} (C))
    = \begin{cases}
        (-1)^{\dim C} & \text{ if } C \text{ is a linear subspace,}
        \\
        0 & \text{ else.}
      \end{cases}
\end{align*}
\end{cor}

We conclude this section by presenting the formula which is an analogue of one obtained by Gr\"unbaum~\cite[Theorem 3.3.]{bG68} for the original Grassmann angles.
\begin{thm}[Gr\"unbaum's formula for modified Grassmann angles]\label{827}
Let $P\subset\R^d$ be an arbitrary convex polytope of full dimension and $1 \leq k\leq d-1$. Then
\begin{align*}
2\sum_{j=0}^{d-1} (-1)^{j} \sum_{F \in \mathcal{F}_{j}(P)}
\sum_{n=0}^{k-1} (-1)^{n}\alpha_{d-k+n}(T_{F}(P))
= (-1)^{d-k} -  (-1)^{d}.
\end{align*}
\end{thm}

\subsection{Expected angles of Gaussian convex cones}
Let $k\in \{1,\dots,d\}$ be fixed.
Consider a random linear operator $A:\R^d\to\R^k$ whose matrix, also denoted by $A$, is given by
\begin{equation*}
A:=\left(
\begin{array}{ccc}
N_{11}&\dots &N_{1d}\\
\vdots &\cdots&\vdots\\
N_{k1}&\dots &N_{kd}\\
\end{array}
\right) \in\R^{ k\times d},
\end{equation*}
where $N_{11},\dots, N_{kd}$ are independent standard Gaussian random variables. \\
For $M \subseteq \R^d$ the set
\[
AM:=\{Ax:x\in M\}\subset \R^k
\]
is called  the \textit{Gaussian image} (or spectrum) of $M$. 
In~\cite[Corollary 3.7.]{DZgotze19} it was found the following connection between the expected angle of the Gaussian image of a cone and its Grassmann angles:
\begin{align}\label{907}
\mathbb{E} [\upsilon_k (AC)] &= \frac{\gamma_k(C) + \gamma_{k-1} (C)} {2}
\end{align}
provided that $C$ is not a linear subspace. Note that since $AC\subset\R^k$, then $\upsilon_k (AC)=\alpha (AC)$ if $AC$ is of full dimension and $\upsilon_k (AC)=0$ otherwise.

With the modified Grassmann angles, 
it possible to obtain~\eqref{907} for arbitrary convex cones:
\begin{align}\label{912}
  \mathbb{E} [\upsilon_k (AC)]  = \alpha_{k-1} (C).
\end{align}
Indeed, due to Theorem~\ref{connection} it suffices to check~\eqref{912} for the linear subspaces only.\\ 
Let $C = \mathbb{R} ^n $, where $ 0\leq n \leq d$.\\
It is known (see \cite[Proposition 5.7]{DZgotze19}) that for any $k \in \N$ and for arbitrary cone $C\subseteq \R^d$,
\[
\mathbb{P}[\dim A C=\min(k,\dim C)]=1.
\]
Consider two cases:
\begin{enumerate}
    \item Case 1: $n \leq k-1.$ \\
Then $\upsilon_k(AC) = 0 $ with probability 1, because $\mathbb{P}[\dim AC  = \min (\dim C , k ) = n] = 1.$\\
So, $\mathbb{E} [\upsilon_k (AC)]=0 $.\\
On the other hand, $\alpha_{k-1} (C) = 0 $, because $\dim C + (d-k+1) = n + (d-k+1) \leq d $.
    \item Case 2: $n \geq k $. \\
By the same argument, \\
$\upsilon_k (AC) = 1 $ with probability 1, because $\mathbb{P}[\dim AC  = \min ( \dim C , k ) = \min ( n , k ) = k ]$,   and $\mathbb{E} [\upsilon_k (AC)]=1 $;
\\
$\alpha_{k-1} (C) = 1 $, because $\dim C + (d-k+1) = n + (d-k+1) > d $.\\
\end{enumerate} 
So,  $$\mathbb{E} [\upsilon_k (A\mathbb{R}^n)] =  \sum_{i\geq 0}  \upsilon_{k+i} (\mathbb{R}^n) = \alpha_{k-1} (\mathbb{R}^n).$$
This completes the proof of~\eqref{912}.

\subsection{Proofs}\label{2247}
\begin{proof}[Proof of Theorem~\ref{connection}]
Note that from the definition of $ W_ {k} ^ {-}, W_ {k} ^ {+} $, as well as from the properties of probability, it follows that:
    \begin{align}\label{e1}
        &\mathbb{P} [W_{k}^{+} \cap C \neq  \{0\} ] =  \mathbb{P} [W_{k}^{-} \cap C \neq  \{0\} ].
        \\
    \nonumber
    &\mathbb{P} [W_{k}^{+} \cap C \neq  \{0\} ] +  \mathbb{P} [W_{k}^{-} \cap C \neq  \{0\} ]  \\ 
    \nonumber
    &=\mathbb{P} \left[ (W_{k}^{+}\cap C\neq  \{0\}) \bigcup (W_{k}^{-}  \cap C\neq  \{0\} )\right] + \mathbb{P} \left[ (W_{k}^{+}\cap C\neq  \{0\}) \bigcap (W_{k}^{-}  \cap C\neq  \{0\} )\right]
    \\
    \nonumber
    &=\mathbb{P} [ (W_{k}^{+}\cup W_{k}^{-} ) \cap C \neq  \{0\} ] + \mathbb{P} \left[ (W_{k}^{+}\cap C\neq  \{0\}) \bigcap (W_{k}^{-}  \cap C\neq  \{0\} )\right] 
    \\
   \label{e2}
    &= \mathbb{P} [ W_{k} \cap C \neq  \{0\} ] + \mathbb{P} \left[ (W_{k}^{+}\cap C\neq  \{0\}) \bigcap (W_{k}^{-}  \cap C\neq  \{0\} )\right]. 
\end{align}
Taking into account Lemma~\ref{633}, we obtain
\begin{align*}  
\mathbb{P} \left[ (W_{k}^{+}\cap C\neq  \{0\}) \bigcap (W_{k}^{-}  \cap C\neq  \{0\} )\right] =  \mathbb{P} \left[W_{k-1} \cap C \neq  \{0\} \right].
\end{align*}
Combining (\ref{e1}) and (\ref{e2}), we get the required.

\end{proof}

\begin{proof}[Proof of Lemma~\ref{633}]
Let us prove (\ref{connection^}) by proving the two following inequalities:
\begin{align}
    \label{geq}\mathbb {P} \left [(W_ {k} ^ {+} \cap C \neq \{0 \}) \bigcap (W_ {k} ^ {-} \cap C \neq \{0 \}) \right] \geq \mathbb {P} \left [W_ {k-1} \cap C \neq \{0 \} \right],\\
    \label{leq}\mathbb {P} \left [(W_ {k} ^ {+} \cap C \neq \{0 \}) \bigcap (W_ {k} ^ {-} \cap C \neq \{0 \}) \right] \leq \mathbb {P} \left [W_ {k-1} \cap C \neq \{0 \} \right].
\end{align}
Let us start with \eqref{geq}. From the definition of half-subspaces it is clear that $W_{k}^{+}\cap W_{k}^{-}= W_{k-1}$, which means that if the cone intersects $W_{k-1}$, then both $W_{k}^{+}$  and  $W_{k}^{-}$ are intersect  with the cone nontrivially. Thus, $$  \mathbb{P} \left[ (W_{k}^{+}\cap C\neq  \{0\}) \bigcap (W_{k}^{-}  \cap C\neq  \{0\} )\right] \geq  \mathbb{P} \left[W_{k-1} \cap C \neq  \{0\} \right]. $$ 
To show \eqref{leq} let us assume that the intersection of the cone with both half-subspaces is nontrivial. Then there are points $ x \neq 0, y \neq 0 $ such that $ x \in W_ {k} ^ {+} \cap C, y \in W_ {k} ^ {-} \cap C. $ \ \
If at least one of the points $ x \text { or } y $ lies in $ W_ {k-1} $, then the cone intersects $ W_ {k-1} $ in a nontrivial way, hence the inequality holds. \\
Therefore, we can assume that in $ W_ {k} $ the points $ x \text { and } y $ are separated by the subspace $ W_ {k-1} $. \\
By the convexity condition, the entire segment $xy$ lies inside the cone. \\
Consider two cases:
\begin{enumerate}
    \item If $x$ and $y$ are not collinear, then $ xy $ intersects $ W_ {k-1} $ at a point different from 0.  Hence, in this case, the intersection of the cone with $ W_ {k-1 } $ is nontrivial. 
    \item If $ x, y, 0 $ are on some line $ l $, then this whole line lies inside $C$ and does not lie inside $ W_ {k-1} $. 
 
In the case when $ C \cap W_ {k} = l \cup A $, where $ A \neq \emptyset $, consider a point $ a \in A $. If $ a \in W_ {k-1} $, then the intersection of $ C $ with $ W_ {k-1} $ is nontrivial, which is what we need. If $ a \notin W_ {k-1} $, then $ a $ lies in one of the open half-spaces $ W_ {k} ^ {+}, W_ {k} ^ {-} $.
Let us connect $ a $ with one of the two points $ x \text { or } y $, which lies in the other half-subspace. Without loss of generality, we can assume that this is the point $ x $. Then the segment $ ax $ lies in the cone $ C $ and intersects $ W_ {k-1} $ at a point different from $0$. 

Thus, it remains to consider the case when $ C \cap W_ {k} = l $. \\
We represent $ C $ as $ C = L \cup \Tilde {C}, $ where $ L $ is the largest linear subspace contained in the cone $ C $. 
\begin{enumerate}
    \item Let $ \dim L \leq d - k $. In this case, $ \mathbb {P} [C \cap W_ {k} = l] = 0 $, since $ C \cap W_ {k} = l $  implies that $ L \cap W_ {k} = l $, but $ \mathbb {P} [L \cap W_ {k} = l] = 0 $. So the left and right hand sides of (\ref{connection^}) are equal to $0$.
    \item Let $ \dim L \geq d-k + 2 $, then $ C \cap W_ {k} $ contains a two-dimensional plane with probability 1, and this contradicts our assumption.
    \item Finally, we need to consider the last option, when $ \dim L = d-k + 1 $. \\
Under the conditions of the lemma, $C \neq \mathbb{R}^{d-k+1} $.
Hence, there is $z\in  \Tilde{C} = C \setminus L$.
Consider the linear hull of $L$ and $z$, denote it by $ \Tilde{L} 
$. So, $\dim \Tilde{L} = d-k+2.$ \\
It follows that $\Tilde{L} \cap W_{k-1} $ contains some line $ \Tilde {l} $ with probability $1$.
The line $\Tilde{l}$ can be represented as $\Tilde{l} =\lambda w$ for some vector $w\in \Tilde{L}$ and $ \lambda \in \mathbb{R}$.
Since $w= \lambda_{1}z+ \lambda_{2}v$ for some $\lambda_{1},\lambda_{2} \in \mathbb{R}, v\in L$, we have  $\Tilde{l} =\lambda w=\lambda (\lambda_{1}z+ \lambda_{2}v)$, where $v\in L, \lambda \in \mathbb{R}$.\\
If $\lambda_{1}=0$, then $\Tilde{l} =\lambda\lambda_{2}v = \Tilde{\lambda} v$, that is, $\Tilde{l}$ lies in $L$, which means that the intersection of $ C $ with $ W_ {k-1} $ is nontrivial.
If $\lambda_{1}\neq0$, then $\Tilde{l} =\lambda (\lambda_{1}z+ \lambda_{2}v) = \lambda\lambda_{1} (z+ \frac{\lambda_{2}}{\lambda_{1}}v) = \Tilde{\lambda} ( z + \Tilde{v})$, and for nonnegative $ \Tilde {\lambda} $, points of the form $ \Tilde {\lambda } (z + \Tilde {v}) $ lie in the cone. It follows that the intersection of $ C $ with $ W_ {k-1} $ is nontrivial. 
\end{enumerate}
\end{enumerate} 

\end{proof}

\begin{proof}[Proof of Theorem~\ref{Crofton}]
Let us consider three cases: 
\begin{enumerate}
    \item If $ C $ is not a linear subspace in $ \mathbb {R} ^ d $, then using the Theorem \ref{connection}  and Crofton formula \eqref{1818}, we obtain the following chain of equalities:
\begin{align*} 
\mathbb{P} [W_{d-k}^{+} \cap C \neq  \{0\} ] =  \frac{1}{2}\left( \mathbb{P} [W_{d-k} \cap C \neq  \{0\} ] + \mathbb{P} [W_{d-k-1} \cap C \neq  \{0\} ] \right)
\\= \frac{1}{2} \left(2 \sum_{i\geq 1  \text{ odd}} \upsilon_{k+i} (C) + 2 \sum_{j\geq 1  \text{ odd}} \upsilon_{k+1+j} (C) \right) = \sum_{i\geq 1}  \upsilon_{k+i} (C).
\end{align*}
\item If $C$ is a linear subspace of dimension $ n $ in $ \mathbb {R} ^ d $, but at the same time $ C \neq \mathbb {R} ^ {k + 1} $, then:
\begin{enumerate}
    \item Suppose $k < n$; then $n+d-k > d$, moreover, since $n \neq k+1 $, we see that  $n\geq k+2$, therefore, by Theorem \ref{connection},
\begin{gather*}
\mathbb{P} [W_{d-k}^{+} \cap C \neq \{0\} ]
\\
= \frac{1}{2} \left( \mathbb{P} [W_{d-k} \cap C \neq \{0\} ]
+ \mathbb{P} [W_{d-k-1} \cap C \neq \{0\} ] \right)
\\
= \frac{1}{2} ( 1+1 ) = 1.
\end{gather*}
 On the other hand, $\slim_{i\geq 1}  \upsilon_{k+i} (C) = 1$, since $ \upsilon_{n}(C) = 1 , 
     \upsilon_{j}(C) = 0 \text{ for } j\neq n$.
     \item 
     Suppose $k \geq n$; then $n+d-k \leq d$, hence,
\begin{gather*} 
\mathbb{P} [W_{d-k}^{+} \cap C \neq  \{0\} ] 
\\
=  \frac{1}{2} \left( \mathbb{P} [W_{d-k} \cap C \neq  \{0\} ] + \mathbb{P} [W_{d-k-1} \cap C \neq  \{0\} ] \right)
\\
= \frac{1}{2} ( 0+0 ) = 0.
\end{gather*}
At the same time, $\slim_{i\geq 1}  \upsilon_{k+i} (C) = 0$, since $ \upsilon_{n}(C) = 1 , 
     \upsilon_{j}(C) = 0 \text{ for } j\neq n$.
\end{enumerate}
\item Finally, consider the case when $ C = \mathbb{R}^{k+1} $. 
In this case, on the one hand, $d-k+k+1=d+1 > d$ and therefore,
$$ \mathbb{P} [W_{d-k}^{+} \cap C \neq  \{0\} ] = 1. $$
On the other hand, $\slim_{i\geq 1}  \upsilon_{k+i} (C) = 1$, because $ \upsilon_{k+1}(C) = 1 , 
     \upsilon_{j}(C) = 0 \text{ for } j\neq k+1.$
\end{enumerate}
     \end{proof}

\begin{proof}[Proof of Theorem~\ref{827}]
For the proof, it is convenient to denote by $\alpha_{k,F,P}$ the $k$-th modified Grassmann  angle for a tangent cone of $P$ at face $F$, i.e., $\alpha_{k,F,P}:= \alpha_{k}(T_{F}(P))$.  

According to Gr\"unbaum \cite{bG68} (with the slightly different notation),  we introduce: 
\begin{align*}
  &\gamma^{k,d} (C^{r}) := \mathbb{P} [W_{k} \cap C^{r} \neq  \{0\} ],
\\
&\sigma^{k,d} (C^{r}) : = 1- \gamma^{k,d} (C^{r}) = \mathbb{P} [W_{k} \cap C^{r} =  \{0\} ],  
\end{align*}
where $ C ^ {r} \subseteq \R^d $ is a convex cone of dimension $ 1\leq r \leq d $  and $W_k$ is a random $k$-dimensional linear subspace having the uniform distribution on
the Grassmann manifold of all such subspaces in $\R^d$.

For a polytope $ P $ and its $ j $-dimensional face $ F ^ {j}$ let
\begin{align*}
    \sigma ^ {k, d} (P, F ^ {j}): = \sigma ^ {k, d} (T_{F^{j}}(P)), 
    \\
    \gamma ^ {k, d} (P, F ^ {j}): = \gamma ^ {k, d} (T_{F^{j}}(P)), 
\end{align*} 
where $ T_{F^{j}}(P) $  was defined in section \ref{polytopes}.

Also, define
$$\sigma_{j}^{k} (P) := \sum_{F^j \in \mathcal{F}_{j}(P)}\sigma^{k,d}(P,F^{j}). $$

In \cite[Theorem 3.3.]{bG68}, Gr\"unbaum proved that for each $ d $-polytope $P$ and $1 \leq k\leq d-1$ the following identity holds:
\begin{align} \label{Gr}
    \sum_{j=0}^{d-k-1} (-1)^{j} \sigma_{j}^{k} (P) = 1 - (-1)^{d-k}.
 \end{align}

We rewrite the last identity in terms of $ \gamma ^ {k, d} (P,F^{j}) = 1- \sigma ^ {k, d} (P,F^{j}) $. First note that the superscript $ d-k-1 $ in the sum on the left hand side can be increased to $ d $, while the value of the sum will not change, since $ \sigma ^ {k, d} (P, F_ {i} ^ {j}) = 0 $ for $ j \geq d-k $.
We get:
\begin{align*}
    \sum_{j=0}^{d-k-1} (-1)^{j} \sigma_{j}^{k} (P) &= \sum_{j=0}^{d} (-1)^{j} \sum_{F^j \in \mathcal{F}_{j}(P)}\sigma^{k,d}(P,F^{j}) 
    \\ 
    &= \sum_{j=0}^{d} (-1)^{j} \sum_{F^j \in \mathcal{F}_{j}(P)} (1-\gamma^{k,d}(P,F^{j}))
    \\
    &=\sum_{j=0}^{d} (-1)^{j} f_{j} - \sum_{j=0}^{d} (-1)^{j} \sum_{F^j \in \mathcal{F}_{j}(P)} \gamma^{k,d}(P,F^{j}) .
 \end{align*}
 
Using Euler's identity (\ref{Eulerp}) for a convex $ d $-dimensional polytope we obtain:
$$ \sum_{j=0}^{d} (-1)^{j} f_{j} - \sum_{j=0}^{d} (-1)^{j} \sum_{F^j \in \mathcal{F}_{j}(P)} \gamma^{k,d}(P,F^{j}) = 1 - \sum_{j=0}^{d} (-1)^{j} \sum_{F^j \in \mathcal{F}_{j}(P)} \gamma^{k,d}(P,F^{j}). $$
Then, comparing the last equality with equality (\ref{Gr}), we can conclude that \begin{align}\label{gammagr}
 \sum_{j=0}^{d} (-1)^{j} \sum_{F^j \in \mathcal{F}_{j}(P)} \gamma ^{k,d}(P,F^{j}) = (-1)^{d-k}.
 \end{align} 
Further, using the definitions of $\alpha_{k, F^{j} , P };  \gamma^{k,d}(P,F^{j})$ and Theorem \ref{connection}, we obtain 
for $F^{j} \neq P  $ and $ 1\leq k \leq d-1 $: 
\begin{align}\label{connection1}
 \alpha_{k, F^{j} , P }  = \frac{1}{2} \left(\gamma^{d-k,d}(P,F^{j}) + \gamma^{d-k-1,d}(P,F^{j}) \right).
\end{align}
Substituting $d-1$ for $k$ in (\ref{connection1}), we get 
\begin{align*}
 \alpha_{d-1, F^{j} , P}  = \frac{1}{2} \left(\gamma^{1,d}(P ,F^{j}) + 0 \right). \end{align*}
Hence,
\begin{align*}
   \gamma^{1,d}(P ,F^{j}) = 2  \alpha_{d-1, F^{j} , P }.
\end{align*}
Repeating the argument above, we get
\begin{gather*}
\alpha_{d-2, F^{j} , P }  = \frac{1}{2} \left(\gamma^{2,d}(P, F^{j}) + \gamma^{1,d}(P, F^{j}) \right) 
\\
\gamma^{2,d}(P, F^{j}) = 2  \alpha_{d-2, F^{j} , P  } - 2\alpha_{d-1, F^{j} , P }
\\
\alpha_{d-3, F^{j} , P  }  = \frac{1}{2} \left(\gamma^{3,d}(P, F^{j}) + \gamma^{2,d}(P, F^{j}) \right)
\\
\gamma^{3,d}(P, F^{j}) = 2  \alpha_{d-3,  F^{j} , P  } - 2\alpha_{d-2,  F^{j} , P  } + 2\alpha_{d-1,  F^{j} , P  }
\\
\cdots
\end{gather*}
 Continuing this line of reasoning, we see that
 \begin{align*}
     \gamma^{k,d}(P, F^{j}) = 2\sum_{n=0}^{k-1} (-1)^{n}\alpha_{d-k+n, F^{j} , P}.
 \end{align*}
Finally, substitute the last identity in (\ref{gammagr}) (note that  for $F^{j} =P $ we have $\gamma^{k,d}(P ,P) = 1 $)  to obtain
\begin{align*}
2\sum_{j=0}^{d-1} (-1)^{j} \sum_{F^j \in \mathcal{F}_{j}(P)}
\sum_{n=0}^{k-1} (-1)^{n}\alpha_{d-k+n, F^{j} , P}
= (-1)^{d-k} -  (-1)^{d}.
\end{align*}
\end{proof}

\section{Proofs of main results}\label{1300}
\subsection{Theorem~\ref{2000}: Properties of discrete intrinsic volumes }\label{809}  
\begin{proof}
Since $\alpha(N_F(P)) = \upsilon_k(T_F(P))$ we can rewrite $A_{k}(P)$ in the following form: 
\begin{align}\label{valuation formula}
A_{k}(P) := \slim_{v \in P\cap \Z^d} \slim_{F \in \mathcal{F}_{k}(P)} \mathbbm {1}[v\in F]\det(F)\upsilon_{k}(T_{F}(P)) \alpha(T_{v}(F)).
\end{align}
Obviously, $A_{k}(\emptyset) = 0.$
First, we must show that for all $P, Q\in\mathcal P(\Z^d)$ such that $P\cup Q \in\mathcal P(\Z^d) $
\begin{align*}
   A_{k}(P\cup Q) = A_{k}(P)+ A_{k}(Q)- A_{k}(P\cap Q),
\end{align*}
that is, 
\begin{gather*}
   \slim_{v \in (P\cup Q) \cap \Z^d}\slim_{F \in \mathcal{F}_{k}(P\cup Q)} \mathbbm {1}[v\in F]\det(F)\upsilon_{k}(T_{F}(P\cup Q)) \alpha(T_{v}(F))) 
   \\
   =  \slim_{v \in P\cap \Z^d} \slim_{F \in \mathcal{F}_{k}(P)} \mathbbm {1}[v\in F]\det(F)\upsilon_{k}(T_{F}(P)) \alpha(T_{v}(F)) 
   \\
   +  \slim_{v \in Q\cap \Z^d} \slim_{F \in \mathcal{F}_{k}(Q)} \mathbbm {1}[v\in F]\det(F)\upsilon_{k}(T_{F}(Q)) \alpha(T_{v}(F)) 
   \\
   - \slim_{v \in (P\cap Q)\cap \Z^d} \slim_{F \in \mathcal{F}_{k}(P\cap Q)} \mathbbm {1}[v\in F]\det(F)\upsilon_{k}(T_{F}(P\cap Q)) \alpha(T_{v}(F)).
\end{gather*}
We will prove that the equality holds for every $v \in (P\cup Q) \cap \Z^d$. In other words for a fixed $v\in P\cup Q$ we have
\begin{gather}\label{for a fixed v}
      \slim_{F \in \mathcal{F}_{k}(P\cup Q)} \mathbbm {1}[v\in F]\det(F)\upsilon_{k}(T_{F}(P\cup Q)) \alpha(T_{v}(F))
   \\ \nonumber
   =    \slim_{F \in \mathcal{F}_{k}(P)} \mathbbm {1}[v\in F]\det(F)\upsilon_{k}(T_{F}(P)) \alpha(T_{v}(F))
   \\ \nonumber
   +    \slim_{F \in \mathcal{F}_{k}(Q)} \mathbbm {1}[v\in F]\det(F)\upsilon_{k}(T_{F}( Q)) \alpha(T_{v}(F))
   \\ \nonumber
   -    \slim_{F \in \mathcal{F}_{k}(P\cap Q)} \mathbbm {1}[v\in F]\det(F)\upsilon_{k}(T_{F}(P\cap Q)) \alpha(T_{v}(F)).
\end{gather}
\begin{defn}
We say that two $k$-dimensional polytopes are collinear if they are in the same $k$-dimensional subspace.
\end{defn}
\noindent Consider some cases.

\textbf{Case 1.} $v\notin P\cap Q$. Without loss of generality, $v\in P\setminus Q$. Then \eqref{for a fixed v} becomes
\begin{gather*}
     \slim_{F \in \mathcal{F}_{k}(P\cup Q)} \mathbbm {1}[v\in F]\det(F)\upsilon_{k}(T_{F}(P\cup Q)) \alpha(T_{v}(F))
   \\
   =   \slim_{F \in \mathcal{F}_{k}(P)} \mathbbm {1}[v\in F]\det(F)\upsilon_{k}(T_{F}(P)) \alpha(T_{v}(F)),
\end{gather*}
which is trivial, since $\upsilon_k$ and $\alpha$ are both defined by a small neighborhood of $v$, and polytopes $P\cup Q$, $P$ are identical in the neighborhood of $v$.

\textbf{Case 2.} $v\in P\cap Q$. To deal with this case we need to describe the connections between faces of $P, Q$ and faces of $P\cup Q, P\cap Q$. We will need two following lemmas.
\begin{lm}\label{unionf}
Let $F$ be a $k$-dimensional face of $P\cup Q$ or $P\cap Q$. Then there exists a $k$-dimensional face $\Tilde{F}$ of $P$ or of $Q$ such that  $F$ and $\Tilde{F}$ are collinear.
\end{lm}
\begin{proof}[Proof of Lemma \ref{unionf}]
Let us consider two cases: 
\begin{enumerate}
    \item $F$ is a $k$-dimensional  face of $P \cup Q$. By definition (see Subsection \ref{Polyhedral sets}), the face $F$ can be represented in the form $F = H_F\cap (P\cup Q)$, where $H_F$ is a supporting hyperplane of $P\cup Q$. In this case, $ H_F\cap P$ and $ H_F\cap Q$ are faces of $P$ and $Q$ respectively. Now we show that both of these faces are $k$-dimensional. Indeed, if $ H_F\cap P$ and $ H_F\cap Q$ have dimension less than $k$, then $F = H_F\cap (P\cup Q) = (H_F\cap P)\cup (H_F\cap Q)$ has dimension less than $k$, which contradicts our assumption.
    
    \item $F$ is a $k$-dimensional  face of $P \cap Q$. By the same argument, $F = H_F\cap (P\cap Q)$ for some  supporting hyperplane $H_F$ of $P\cap Q$.
    
    First note that at least one of the polytopes $P$ or $Q$ lies on one side of the $H_F$.
    In fact, by convexity of the union of $P$ and $Q$, if $p \in P$, $q \in Q$,  then the segment $pq$ contains at least one point from $P\cap Q$. So, if we assume that there are points $p_1$ and $q_1$ on one side of the $H_F$, $p_2$ and $q_2$ on the other side of the $H_F$, then, according to the above, there exists $x \in p_1q_1, y \in p_2q_2$, such that  $x, y \in P \cap Q$. Consequently,  $P \cap Q$ does not lie on one side of $H_F$, which contradicts the assumption.
    
    We can assume without loss of generality that $P$ lies on one side of the $H_F$.  Then $H_F \cap P$ is a face of $P$.  It remains to check that $H_F \cap P$ is $k$-dimensional.\\
    Assume the converse: the dimension of $H_F \cap P$ at least $k+1$. There are three cases. 
    \begin{enumerate}
        \item \label{(a)} $Q$ does not lie on one side of $H_F$. Since $\dim H_F \cap P \geq k+1$, there is a point $w \notin F$, $w \in  H_F \cap P$, such that $\dim \lin(w, F) = k+1,$ where $\lin(w, F)$ is a linear hull of $w$ and $F$. Let us prove that in this case $w \in Q$. By assumption, there are two points $q_1, q_2 \in Q $, such that $ q_1, q_2  $ lie on opposite sides of $H_F$. Both segments $q_1w$ and $q_2 w$ contain points from $P\cap Q$. But $P\cap Q$ lies on one side of $H_F$. Therefore, the only point on the segments $q_1w,q_2 w$ lying at the $P\cap Q$ is $w$. This
contradicts our assumption that $F = H_F\cap (P\cap Q)$ is $k$-dimensional.
        \item \label{(b)} $Q$ and $P$ lie on opposite sides of $H_F$. This case is analyzed similarly to the previous one.
        \item $Q$ lies on the same side of $H_F$ as $P$. Then $H_F\cap Q$ is a face of $Q$. If $H_F\cap Q$ is $k$-dimensional, then the lemma is proved with $\Tilde{F} = H_F\cap Q$.
        Now suppose that $\dim H_F\cap Q \geq k+1 $. There are a points $w_1 \notin F$, $w_1 \in  H_F \cap P$, $w_2 \notin F$, $w_2 \in  H_F \cap Q$, such that $\dim \lin(w_1, F) =\dim \lin(w_2, F) = k+1$. Suppose that $\dim \lin(w_1, w_2, F) = k+2$. As above, the segment $w_1w_2$ contains a point $z \in P\cap Q$. Moreover, $z \notin F$, since $\dim \lin(w_1, w_2, F) = k+2$. On the other hand, $z \in H_F\cap (P\cap Q) = F $, which is the desired contradiction. If $\dim \lin(w_1, w_2, F) = k+1$, then $w_1$ and $ w_2$ are collinear. In
the case that $w_1$ and $ w_2$ co-directed, we have $\dim H_F \cap (P \cap Q) = k+1$, which contradicts our assumption.
If $w_1$ and $ w_2$ are oppositely
directed, then $F$ is a face of $H_F \cap P$, and hence $F$ is a face of $ P$.
This concludes the proof.
    \end{enumerate}

\end{enumerate}

\end{proof}
\begin{lm} \label{unionff}
Let $E$ be a $k$-dimensional face of $P$, such that there does not exist face of $Q$ collinear to $E$. Then $E$ is a $k$-dimensional face of $P\cap Q$ or of $P\cup Q$.
\end{lm}
\begin{proof}[Proof of Lemma \ref{unionff}] Following the notation of the previous lemma consider a hyperplane $H_E$ such that $E = H_E\cap P$ and consider three cases%
\begin{enumerate}
    \item $Q$ entirely lies on the same side of $H_E$ as $P$. In this case $H_E$ is a supporting hyperplane for $P
    \cup Q$ and $E\subset H_E\cap (P\cup Q)$. Assume that there does not exist a face of $P\cup Q$ collinear to $E$, then $H_E\cap(P\cup Q)$ has dimension at least $k + 1$ and therefore $H_E\cap Q$ is of dimension at least $k + 1$ and consequently $H_E\cap P\subset H_E\cap Q$. Hence $H_E\cap(P\cap Q) = H_E\cap P = E$.
    \item $Q$ does not lie entirely on the one side of $H_E$. In this case by the same argument as in case \ref{(a)} of the previous lemma we get that $E\subset Q$ and E is a face of $P\cap Q$.
    \item $Q$ lies entirely on the opposite side of $H_E$ than $P$. Here proof follows by the same argumetns as in case \ref{(b)}  of Lemma \ref{unionf}.
\end{enumerate}
\end{proof}
With these lemmas in hand, we can prove \eqref{for a fixed v} by independent consideration of collinear classes of faces. Fix  $k$-dimensional face $E$ of $P$.

\noindent\textbf{Case 2.1.} There is no face of $Q$ collinear to E. The second lemma gives that $E$ is a face of $P\cup Q$ or $P\cap Q$.

\noindent\textbf{Case 2.1.1} $E$ is a face of $P\cap Q$. We will need the following.
\begin{lm} \label{convexity of union cones}
    Let $C_1$ and $C_2$ be cones of different dimensions such that their union is convex. Then one of them contains the other.
\end{lm}
\begin{proof}[Proof of Lemma \ref{convexity of union cones}]
Suppose that $\dim C_1 = k, \dim C_2 = n, k>n \text{ and } x \in C_1 \cup C_2.$ By the convexity condition, $x + C_1 \cup C_2 \subseteq  C_1 \cup C_2 $. Hence for every neigborhood $U$ of $x$ 
$U\cap (C_1 \cup C_2)$ is at least $k$-dimensional, hence $x\in C_1$.

\end{proof}
We know that $E\subset Q$ and $E$ is not a face of $Q$. Hence the cone $T_E(Q):= T_e(Q)$ (for some $e \in \relint(E)$) contains a $k$-dimensional subspace. Also, we know that $Q$ has no faces collinear to $E$, therefore $T_E(Q)$ is a cone of dimension at least $k + 1$ while $T_E(P)$ is of dimension exactly $k$.
Lemma \ref{convexity of union cones} gives $T_E(P)\subset T_E(Q)$; therefore $T_E(P)= T_E(P\cap Q)$. Now we see that in \eqref{for a fixed v}, the terms with face $E$ coincide and have different signs: one term is in the sum for $P$ and the other is in the sum for $P\cap Q$.

\noindent\textbf{Case 2.1.2} $E$ is a face of $P\cup Q$. By the same argument as above we conclude that cone $T_e(Q), e\in E,$ is at most $(k-1)$-dimensional and $T_E(P)= T_E(P\cup Q)$; therefore terms for face $E$ again coincide.

\noindent\textbf{Case 2.2} There exists face $E_1$ of $Q$ such that $E_1$ and $E$ are collinear. Then there is a face $E_2 = E \cup E_1$ of $P\cup Q$ collinear to E.

\noindent\textbf{Case 2.2.1} There is no face of $P\cap Q$ collinear to $E$. Consider an arbitrary point $p\in \relint (E)$. It is easy to see that $p\notin Q$. Then,
\begin{align}\label{with point p}
    T_E(P) = T_p(P) = T_p(P\cup Q) = T_{E_2}(P\cup Q).
\end{align} 
By doing the same for $Q$ we establish
\begin{align}\label{cone equality}
    T_E(P) = T_{E_1}(Q) = T_{E_2}(P\cup Q).
\end{align}
Denote the cone in the latter equation by $C$. Now the terms in \eqref{for a fixed v} for a face $E$ are the following:
\begin{align*}
    \det (E_2) v_k(C) \alpha(T_v(E_2)) = \det (E) v_k(C) \alpha(T_v(E)) + \det (E_1) v_k(C) \alpha(T_v(E_1)).
\end{align*}
All three determinants are equal, hence the equality follows from the additivity of the solid angle (in this case $\alpha(T_v(P\cap Q)) = 0$).

\noindent\textbf{Case 2.2.2}. There exists a $k$-dimensional face $E_3 = E\cap E_1$ of $P\cap Q$. In this case we have four polytopes $E, E_1, E_2, E_3$ which can be in a different relations. 

\noindent\textbf{Case 2.2.2.1}. $E\nsubseteq E_1$ and $E_1\nsubseteq E$. Then \eqref{cone equality} holds by the same argument as in the previous case. Also, consider point $r\in \relint  (E_3)\subset \relint (E)$. We have  
\begin{align}\label{with point r}
    T_{E_3}(P\cap Q) = T_{r}(P\cap Q) = T_{r}(P)\cap T_{r}(Q) = T_r(P) = T_E(P).
\end{align}
Therefore,
\begin{align*}
    T_E(P) = T_{E_1}(Q) = T_{E_2}(P\cup Q) = T_{E_3}(P\cap Q).
\end{align*}
If we denote the cone in the latter equality by $C,$ the terms in \eqref{for a fixed v} collinear to $E$ will be
\begin{align*}
    \det (E_2) v_k(C) \alpha(T_v(E_2)) &= \det(E) v_k(C) \alpha(T_v(E)) 
    \\
    &+ \det (E_1) v_k(C) \alpha(T_v(E_1))- \det (E_3) v_k(C) \alpha(T_v(E_3)).
\end{align*}
And again we have equal determinants and additivity of the solid angle.

\noindent\textbf{Case 2.2.2.2} $E_1 = E$. Then $E = E_1 = E_2 = E_3$ and 
\begin{align*}
    \alpha(T_v(E)) = \alpha(T_v(E_1)) =\alpha(T_v(E_2)) = \alpha(T_v(E_3)) = \alpha.
\end{align*}
In this case \eqref{for a fixed v} is reduced to 
\begin{align*}
    v_k(T_{E_2}(P\cup Q)) = v_k(T_{E}(P)) + v_k(T_{E_1}(Q)) - v_k(T_{E_3}(P\cap Q)),
\end{align*}
which follows from the additivity property of the conic intrinsic volume (see \cite[Section 6.5]{SW08}).

\noindent\textbf{Case 2.2.2.3} $E\subset E_1$, $E\neq E_1$. Then $E_2 = E_1$, $E_3 = E$. The same arguments as in \eqref{with point p} and in \eqref{with point r} give
\begin{align*}
    T_{E_1}(Q) = T_{E_2}(P\cup Q), \quad T_{E}(P) = T_{E_3}(P\cap Q).
\end{align*}

Now the terms in \eqref{for a fixed v} for $P$ coincide with the terms for $P\cap Q$ and the terms for $Q$ coincide with the terms for $P\cup Q$.

Thus, $A_{k}(\cdot)$ is a valuation. The translation-invariance of $A_{k}(\cdot)$ is trivial.

 It is easily shown that $A_0 (P) = 1$ and $A_d (P) = A(P)$:
 \begin{gather*}
   A_0(P)=\sum_{v- \text{vertex of } P} \upsilon_0(T_v(P)) = 1,
   \\
   \nonumber
    A_d(P)=\sum_{v\in P\cap \Z^d} \upsilon_d(T_P(P))\alpha(T_v(P)) = \sum_{v\in P\cap \Z^d}\alpha(T_v(P))= A(P).
 \end{gather*}
Here, we write by definition $\alpha(T_v(v))=1, \det(v) = 1$ when $v$ is vertex of $P$, and used the properties of the conic intrinsic volumes, in particular, the fact that for any $P \in \mathcal{P}(\Z^d)$ we have $\slim_{v- \text{vertex of } P} \upsilon_0(T_v(P)) = 1.$

Evidently, $A_k(P) = 0$ in the case when $\dim P< k$, since $P$ has no $k$ - dimensional faces.

The fourth part of Theorem \ref{2000} is equivalent to the first part Theorem \ref{1021}, the proof of which is presented below.
\end{proof}

\subsection{Theorem~\ref{1021}: Properties of intrinsic Ehrhart polynomials}\label{816}
\begin{proof}
Rewrite \eqref{valuation formula} in the following way:
\begin{align*}
    A_{k}(P) &= \slim_{F \in \mathcal{F}_{k}(P)}\slim_{v \in F\cap \Z^d}  \mathbbm {1}[v\in F] \det(F)\upsilon_{k}(T_{F}(P)) \alpha(T_{v}(F))
    \\
    &= \slim_{F \in \mathcal{F}_{k}(P)}\det(F)\upsilon_{k}(T_{F}(P))\slim_{v \in F\cap \Z^d}  \mathbbm {1}[v\in F] \alpha(T_{v}(F))
    \\
    &=\slim_{F \in \mathcal{F}_{k}(P)}\det(F)\upsilon_{k}(T_{F}(P))\slim_{E\in \mathcal{F}(F)}\slim_{v \in \relint(E)}  \alpha(T_{v}(F))
    \\
    &= \slim_{F \in \mathcal{F}_{k}(P)}\det(F)\upsilon_{k}(T_{F}(P))\slim_{E\in \mathcal{F}(F)}L_{\relint(E)}  \alpha(T_{E}(F)).
\end{align*}
Hence,
\begin{align}\label{polynomial A_k,P}
    A_{k, P}(t) &= \slim_{F \in \mathcal{F}_{k}(P)}\det(F)\upsilon_{k}(T_{F}(P))\slim_{E\in \mathcal{F}(F)}L_{\relint(E)}(t)  \alpha(T_{E}(F))
    \\
    \nonumber &= \slim_{F \in \mathcal{F}_{k}(P)}\det(F)\upsilon_{k}(T_{F}(P))  A_F(t).
\end{align}
Here by $A_F(t)$ we mean the sum $\slim_{E\in \mathcal{F}(F)}L_{\relint(E)}(t)  \alpha(T_{E}(F))$ in $k$-dimensional space in $\R^d$, generated by the face $F$. The results mentioned in section \ref{2329} are correct for $A_F(t)$ with minor changes. By 
\eqref{1010}, \eqref{1011}, $A_F(t)$ is always even or odd polynomial of degree $k$ with zero constant term. Therefore, the same holds for $ A_{k, P}(t)$.

\medskip To calculate the leading coefficient notice that the terms in \eqref{polynomial A_k,P} where $E\neq F$ have degree strictly less than $k$ and consequently do not impact the leading coefficient; the terms $L_{\relint(E)}(t)$ where $E=F$, by \eqref{1435} and \eqref{651},
have the leading coefficient $\frac{|F|}{\det(F)}$. Thus, the leading coefficient of $A_{k, P}(t)$ equals
\begin{align*}
    \slim_{F \in \mathcal{F}_{k}(P)}\det(F)\upsilon_{k}(T_{F}(P))\frac{|F|}{\det(F)} = \slim_{F \in \mathcal{F}_{k}(P)}\upsilon_{k}(T_{F}(P)) |F| \stackrel{\eqref{2107a}}{=} V_k(P).
\end{align*}

Finally we shall show that reciprocity law holds:
\begin{align*}
     A_{k, P}(-t) &= \slim_{F \in \mathcal{F}_{k}(P)}\det(F)\upsilon_{k}(T_{F}(P))  A_F(-t)
     \\
     &\stackrel{\eqref{651}}{=} \slim_{F \in \mathcal{F}_{k}(P)}\det(F)\upsilon_{k}(T_{F}(P))  (-1)^{\dim F}A_F(t)
     \\
     &= (-1)^{k}  A_{k, P}(t).
\end{align*}
\end{proof}

\subsection{Theorem~\ref{1548}: Properties of Grassmann angle valuations}\label{821}
\begin{proof}
By the definition of translation-invariant valuation, we need to show that
\begin{enumerate}
    \item $G_{k}(\emptyset) = 0 ;$
    \item  $G_{k}(P \cup Q)  + G_{k}(P \cap Q)  =  G_{k}(P) + G_{k}(Q)  $ for any $P, Q \in \mathcal{P} ( \mathbb{Z}^d) ;$  
    \item $G_{k}(P+z) = G_{k}(P) $ for any $P \in \mathcal{P} ( \mathbb{Z}^d) $ and $z\in \mathbb{Z}^d$.
\end{enumerate}
The first and third properties are obvious.
Let us check the second one.\\
We must prove that
\begin{align*} 
\sum_{v \in (P\cup Q)\cap\mathbb{Z}^d} \alpha_{k}(T_{v}(P\cup Q)) &+ \sum_{v \in (P\cap Q)\cap\mathbb{Z}^d} \alpha_{k}(T_{v}(P\cap Q)) 
\\
&=  \sum_{v \in P\cap\mathbb{Z}^d} \alpha_{k}(T_{v}(P))+  \sum_{v \in Q\cap\mathbb{Z}^d} \alpha_{k}(T_{v}(Q)).
\end{align*}
Using Theorem \ref{Crofton}  
and the trivial fact that
\begin{align*}
    T_{v}(P\cup Q)=T_{v}(P) \cup T_{v}(Q),
    \\
    T_{v}(P\cap Q)=T_{v}(P) \cap T_{v}(Q),
\end{align*}
we get 
\begin{gather*}
    \sum_{v \in (P\cup Q)\cap\mathbb{Z}^d} \alpha_{k}(T_{v}(P\cup Q)) + \sum_{v \in (P\cap Q)\cap\mathbb{Z}^d} \alpha_{k}(T_{v}(P\cap Q))\\= \sum_{v \in (P\cup Q)\cap\mathbb{Z}^d} \sum_{i\geq 1}  \upsilon_{k+i} (T_{v}(P\cup Q))+\sum_{v \in (P\cap Q)\cap\mathbb{Z}^d} \sum_{i\geq 1}  \upsilon_{k+i} (T_{v}(P\cap Q))\\ =  \sum_{v \in (P\cup Q)\cap\mathbb{Z}^d} \sum_{i\geq 1}  \upsilon_{k+i} (T_{v}(P) \cup T_{v}(Q))+\sum_{v \in (P\cap Q)\cap\mathbb{Z}^d} \sum_{i\geq 1}  \upsilon_{k+i} (T_{v}(P) \cap T_{v}(Q)).
\end{gather*}
By the additivity of the conic intrinsic volumes (see \cite[Section 6.5]{SW08}), we can rewrite the last expression in the following form:
\begin{gather*}
\sum_{v \in (P\cup Q)\cap\mathbb{Z}^d} \sum_{i\geq 1}  \upsilon_{k+i} (T_{v}(P) \cup T_{v}(Q))+\sum_{v \in (P\cap Q)\cap\mathbb{Z}^d} \sum_{i\geq 1}  \upsilon_{k+i} (T_{v}(P) \cap T_{v}(Q)) 
\\
=\sum_{v \in (P\cup Q)\cap\mathbb{Z}^d} \sum_{i\geq 1}  \upsilon_{k+i} (T_{v}(P)) + \upsilon_{k+i} (T_{v}(Q)) - \upsilon_{k+i} (T_{v}(P) \cap T_{v}(Q)) \\+\sum_{v \in (P\cap Q)\cap\mathbb{Z}^d} \sum_{i\geq 1}  \upsilon_{k+i} (T_{v}(P) \cap T_{v}(Q)) 
\\
=\sum_{v \in (P\cup Q)\cap\mathbb{Z}^d} \sum_{i\geq 1}  \upsilon_{k+i} (T_{v}(P)) + \upsilon_{k+i} (T_{v}(Q))\\ =\sum_{v \in P\cap\mathbb{Z}^d} \sum_{i\geq 1}  \upsilon_{k+i} (T_{v}(P)) + \sum_{v \in Q\cap\mathbb{Z}^d} \sum_{i\geq 1}\upsilon_{k+i} (T_{v}(Q)) 
\\
= \sum_{v \in P\cap\mathbb{Z}^d} \alpha_{k}(T_{v}(P))+  \sum_{v \in Q\cap\mathbb{Z}^d} \alpha_{k}(T_{v}(Q)). 
\end{gather*}
This implies that $G_k(\cdot)$ is a translation-invariant valuation on $\mathcal{P}(\Z^d)$.

Further, we see that for full-dimensional $P \in \mathcal{P}(\Z^d),$ 
\begin{align*}
G_{d-1}(P) = \sum_{v \in P\cap\mathbb{Z}^d} \alpha_{d-1}(T_{v}(P)) =  \sum_{v \in P\cap\mathbb{Z}^d} \alpha(T_{v}(P)) = A(P) 
\end{align*} 
is the solid-angle valuation for the polytope $ P \in \mathcal{P}(\Z^d)$.

Note also that since $\alpha_{d}(C)\equiv 0  $ for any cone $C$, we have
$$G_{d}(P)\equiv 0 .$$
The following step is to show that the Ehrhart valuation $L(P)$, up to a constant, coincides with $ G_ {0} (P) $:
\begin{align*}
 G_{0}(P) = \sum_{v \in P\cap\mathbb{Z}^d} \alpha_{0}(T_{v}(P)) =  \sum_{v \in P\cap\mathbb{Z}^d} \sum_{i\geq 1}  \upsilon_{i} (T_{v}(P)) \\ =  \sum_{v \in P\cap\mathbb{Z}^d} 1 -\upsilon_{0} (T_{v}(P)) = L(P) -   \sum_{v \in P\cap\mathbb{Z}^d}\upsilon_{0} (T_{v}(P)). 
\end{align*}
Further, $ \upsilon_ {0} (T_ {v} (P)) = 0 $ when $ v $ is not a vertex of the polytope $ P $. Indeed, 
 $\upsilon_0(T_ {v} (P)) = \mathbb{P}[\Pi_{T_ {v} (P)}(g)\in \relint( \text{of some } 0-\text{ dimensional face of } T_ {v} (P))].$
But if $ v $ is not a vertex of $ P $, then $ T_ {v} (P) $ does not contain $0$ -- dimensional faces.

So,
\begin{align*}
 G_{0}(P) = L(P) -   \sum_{v - \text{ vertex of } P}\upsilon_{0} (T_{v}(P)) = L(P) - 1.
 \end{align*}
In the last step, we  used the 
fact that $\slim_{v - \text{ vertex of } P}\upsilon_{0} (T_{v}(P)) = \slim_{v - \text{ vertex of } P} \alpha (T_{v}(P)^{\polar}) =  1$ for $P \in \mathcal{P} (\Z^d)$.

The fourth part of Theorem \ref{1548} easily follows from inequality \eqref{inequality} and properties of angles $\alpha_k$, including Theorem \ref{connection}.
\end{proof}

\subsection{Theorem~\ref{824}: Properties of Grassmann polynomials}\label{823}
\begin{proof}
Theorem \ref{1548} implies polynomiality of $G_{k,P}(t)$ because of McMullen's result (see \eqref{1621}).
Nevertheless, let us show another standard short reasoning.\\
We can represent the dilated polytope $tP$ as a disjoint union of its relative open faces:
\begin{align*}
 tP = \bigcup_{F\in \mathcal{F} (P)} \relint(tF).
\end{align*}
Hence we can write $$G_{k,P} (t) = \slim_{v \in tP\cap\mathbb{Z}^d} \alpha_{k}(T_{v}(tP)) = \slim_{F\in \mathcal{F} (P)}\slim_{v \in tP\cap\mathbb{Z}^d} \alpha_{k}(T_{v}(tP))\mathbbm{1}[v \in \relint (tF)] .$$
Since  $T_{v_{1}}(tP)=T_{v_{2}}(tP) $ for any $v_{1} , v_{2} \in \relint(tF) $, we see that   $\alpha_{k}(T_{v}(tP))$ is constant on each relatively open face $ \relint (tF)$, and we denote this
constant by $\alpha_{k,F,P} $,  whence, 
\begin{align} \label{polynomialG}
G_{k,P} (t) = \slim_{F\in \mathcal{F}(P) } \alpha_{k,F,P}  \slim_{v \in tP\cap\mathbb{Z}^d} 
\mathbbm{1}[v \in \relint (tF)]
=
\slim_{F\in \mathcal{F}(P) } \alpha_{k,F,P} L_{\relint(F)} (t) .
\end{align}
Thus $ G_ {k, P} (t) $ is a polynomial in $ t $ of degree at most $ d $, because $L_{\relint(F)} 
$ is a polynomial.\\
Next, notice that every term in the right hand side of \eqref{polynomialG} has degree strictly less then $d$, except the one where $F=P$. Hence, the leading coefficient of $G_ {k, P} (t)$ is equal to $\alpha_{k,P,P} |P|$, since the leading coefficient of $L_{\relint(P)} (t)$ is $|P|$, because $P$ is full-dimensional. Moreover,  $\alpha_{k,P,P} = 1 $, therefore $G_{k}^{(d)}(P)=|P|$.

It remains to check that $G_{k}^{(0)}(P)=0.$
\\
For $k=0$ it follows from the fact that constant term of $L_P(t)$ is $1$ and $G_{0,P}(t) = L_P(t) - 1$.\\
Otherwise, first note that from \eqref{polynomialG} 
we have the relation 
$$ G_{k,P} (0) = \slim_{F\in \mathcal{F}(P) } \alpha_{k,F,P} L_{\relint(F)} (0) =\slim _{F\in \mathcal{F}(P) } \alpha_{k,F,P} (-1)^{dim F}. $$
Here the last equality holds due to the fact that $ L_ {P} (0) = 1 $ and due to the Ehrhart–Macdonald reciprocity law, which was mentioned in the introduction: 
$$ L_{P} (-t) = (-1)^{dim P}L_{\relint(P)} (t). $$ 
So, to express the constant term of the polynomial $ G_ {k, P} $, we need to understand what is the sum $ \slim_ {F\in \mathcal{F}(P)} \alpha_ {k, F, P} (-1) ^ {dim F} .$
Again, as in the proof of Theorem \ref{827}, we apply formula (\ref{connection1}) to obtain
\begin{gather*}
 \slim_ {F\in \mathcal{F}(P)} \alpha_ {k, F, P} (-1) ^ {dim F} = \sum_{j=0}^{d} (-1)^{j} \sum_{F^j \in \mathcal{F}_{j}(P)} \alpha_{k, F^{j}, P}   \\
= \sum_{j=0}^{d-1} (-1)^{j} \sum_{F^j \in \mathcal{F}_{j}(P)} \frac{1}{2}\left(\gamma^{d-k,d}(P,F^{j})+ \gamma^{d-k-1,d}(P,F^{j})\right)  + (-1)^{d} \alpha_{k, P, P}. 
\end{gather*}
Next, we use formula (\ref{gammagr}) and the fact that $\alpha_{k, P, P} =1$  and $\gamma^{d-k,d}(P,P)= \gamma^{d-k-1,d}(P,P)=1$ to obtain
\begin{gather*}
\frac{1}{2}\sum_{j=0}^{d-1} (-1)^{j} \sum_{F^j \in \mathcal{F}_{j}(P)} \gamma^{d-k,d}(P,F^{j})+\frac{1}{2}\sum_{j=0}^{d-1} (-1)^{j} \sum_{F^j \in \mathcal{F}_{j}(P)} \gamma^{d-k-1,d}(P,F^{j})  + (-1)^{d} \\= \frac{1}{2} \left( (-1)^{k} - (-1)^{d}\gamma^{d-k,d}(P,P)\right) + \frac{1}{2} \left( (-1)^{k+1} - (-1)^{d}\gamma^{d-k-1,d}(P,P)\right) +(-1)^{d}\\ = \frac{1}{2} \left( (-1)^{k} - (-1)^{d}\right) + \frac{1}{2} \left( (-1)^{k+1} - (-1)^{d}\right) +(-1)^{d} = 0.
\end{gather*}
Thus, $\slim_{F\in \mathcal{F}(P)} \alpha_{k,F,P} (-1)^{dim F} =0$, which completes the proof.
\end{proof}

 In the process of proving Theorem \ref{824} we established the following analogue of the Brianchon--Gram relation \eqref{Brianchon} for the modified Grassmann angles.
\begin{st}
    If $P \in \mathcal{P} (\Z^d)$ is $d$-dimensional polytope, then
\begin{align*}
  \slim_{F\in \mathcal{F}(P)} \alpha_{k,F,P} (-1)^{dim F} =0.
\end{align*}
\end{st}



\subsection{Theorem~\ref{938}: Grassmann polynomials of Reeve's tetrahedron}\label{943}





\begin{proof}
Let us find the coefficients of the polynomial $G_{1, \Delta_{h}}(t)$:
\begin{align*}
G_{1, \Delta_{h}}(t) = \sum_{v \in t\Delta_{h}\cap\mathbb{Z}^d} \alpha_{1}(T_{v}(t\Delta_{h})).
\end{align*}
To this end, we need to recall the definition of the polar cone (see Section \ref{Convex cones}) and prove the following statement.
Let $ n $ denote the normal to the random 
subspace $ W_ {d-1} $ in $ d $-dimensional space, with a fixed direction, passing through $ 0 $.
\begin{st}\label{polar}
For $ P \in \mathcal{P} (\mathbb{Z}^{d}) $ and $v \in  P\cap\mathbb{Z}^d $ the following two
conditions are equivalent:
\begin{enumerate}
    \item[(a)] $T_{v}(P) \cap W_{d-1} = \{0\};$
    \item[(b)] $\big\{ n \cap \relint( T_{v}(P)^{\polar}) \big\} \bigcup \big\{ -n \cap \relint( T_{v}(P)^{\polar}) \big\}  \neq \{0\}.$
\end{enumerate}
\end{st}
\begin{proof}
The proof is a chain of equivalent transitions:
\begin{align*}
T_{v}(P) \cap W_{d-1} = \{0\} \Leftrightarrow \big\{\forall x \in T_{v}(P) \text{  } \langle x,n \rangle > 0 \big\} \bigcup \big\{\forall x \in T_{v}(P) \text{  } \langle x,n \rangle < 0 \big\}  \Leftrightarrow\\ 
\big\{ n \cap \relint( T_{v}(P)^{\polar}) \big\} \bigcup \big\{ -n \cap \relint( T_{v}(P)^{\polar}) \big\}  \neq \{0\}.
\end{align*} 
\end{proof}
Denote by $l$ the line containing the normal $ n $. Then
\begin{align*}
\mathbb{P} [W_{d-1}\cap T_{v}(P) \neq  \{0\} ] &= 1 - \mathbb{P} [W_{d-1} \cap T_{v}(P) =  \{0\} ]
\\
&= 1 - \mathbb{P}\big[\big\{ n \in \relint( T_{v}(P)^{\polar}) \big\} \bigcup \big\{ -n \in \relint( T_{v}(P)^{\polar}) \big\} \big]
\\
&= 1 - \mathbb{P}[ l \cap \relint( T_{v}(P)^{\polar}) \neq   \{0\} ].
\end{align*}
Further, we apply  Theorem \ref{connection} and Statement \ref{polar} to $P\in \mathcal{P} (\mathbb{Z}^{3}) $ with $\dim P = 3$: 
\begin{align*} 
G_{1, P}(t) &= \sum_{v \in tP\cap\mathbb{Z}^3} \alpha_{1}(T_{v}(tP)) = \sum_{v \in tP\cap\mathbb{Z}^3} \mathbb{P} [W_{2}^{+} \cap  (T_{v}(tP)) \neq  \{0\} ] 
\\
&= \sum_{v \in tP\cap\mathbb{Z}^3} \frac{1}{2} \left( \mathbb{P} [W_{2} \cap (T_{v}(tP))  \neq  \{0\} ] + \mathbb{P} [W_{1} \cap  (T_{v}(tP)) \neq  \{0\} ]\right)
\\
&= \frac{1}{2}\left(\sum_{v \in tP\cap\mathbb{Z}^3} 1 -  \sum_{v \in tP\cap\mathbb{Z}^3} \mathbb{P}[ l \cap  \relint(T_{v}(tP)^{\polar}) \neq   \{0\} ] \right) 
\\
&+ \frac{1}{2}\sum_{v \in  tP\cap\mathbb{Z}^3} \mathbb{P} [W_{1} \cap  (T_{v}(tP)) \neq  \{0\} ].
\end{align*}
In the case of Reeve’s tetrahedron, we have
\begin{gather}
\label{Reeve} G_{1, \Delta_{h}}(t) = \sum_{v \in t\Delta_{h}\cap\mathbb{Z}^3} \alpha_{1}(T_{v}(t\Delta_{h})) = \sum_{v \in t\Delta_{h}\cap\mathbb{Z}^3} \mathbb{P} [W_{2}^{+} \cap (T_{v}(t\Delta_{h})) \neq  \{0\} ] 
\\
\nonumber =  \sum_{v \in t\Delta_{h}\cap\mathbb{Z}^3}\frac{1}{2} \left( \mathbb{P} [W_{2} \cap (T_{v}(t\Delta_{h}))  \neq  \{0\} ] + \mathbb{P} [W_{1} \cap  (T_{v}(t\Delta_{h})) \neq  \{0\} ] \right)
\\
\nonumber =\frac{1}{2} \big(\sum_{v \in t\Delta_{h}\cap\mathbb{Z}^3} 1 -  \sum_{v \in t\Delta_{h}\cap\mathbb{Z}^3} \mathbb{P}[ l \cap \relint( T_{v}(t\Delta_{h})^{\polar}) \neq   \{0\} ] \text{ } \big) + \frac{1}{2} \sum_{v \in t\Delta_{h}\cap\mathbb{Z}^3} \mathbb{P} [W_{1} \cap  (T_{v}(t\Delta_{h})) \neq  \{0\} ] 
\\
\nonumber  = \frac{1}{2} L_{\Delta_{h}}(t) - 1 + A_{\Delta_{h}}(t) -\frac{1}{2} L_{\relint (\Delta_{h})} (t)
\\ \nonumber
= \frac{1}{2}\left(\frac{h}{6}t^3 + t^{2} +\left(2 - \frac{h}{6} \right)t +1\right) -1 +\frac{h}{6}t^3 +  \left( S - \frac{h}{6} \right) t 
 - \frac{1}{2} \left( \frac{h}{6}t^3 - t^{2} +\left(2 - \frac{h}{6} \right)t -1\right) 
 \\
 =  \nonumber \frac{h}{6}t^3 + t^2 + \left( S - \frac{h}{6} \right) t .
\end{gather}
Here the fifth equality follows from the fact that
$$\sum_{v \in t\Delta_{h}\cap\mathbb{Z}^3}\mathbb{P}[ l \cap \relint( T_{v}(t\Delta_{h})^{\polar}) \neq   \{0\} ] \text{ }  = \sum_{v -\text{ vertex of } t\Delta_{h}} \mathbb{P}[ l \cap \relint( T_{v}(t\Delta_{h})^{\polar}) \neq   \{0\} ] \text{ }.  $$
Indeed, if ${v \in t\Delta_{h}\cap\mathbb{Z}^3}$ is not a vertex of $t\Delta_{h}$, then the polar cone lies in a linear subspace of dimension at most 
$2$, hence $\mathbb{P}[ l \cap  \relint( T_{v}(t\Delta_{h})^{\polar}) \neq   \{0\} ]=0.$\\
 Moreover, from \eqref{sol2} and the identity $$\slim_{v - \text{ vertex of } P} \alpha (T_{v}(P)^{\polar}) =  1$$ for $P \in \mathcal{P} (\Z^d)$ it easily follows that
$$\sum_{v -\text{ vertex of } t\Delta_{h}} \mathbb{P}[ l \cap \relint( T_{v}(t\Delta_{h})^{\polar})  \neq   \{0\} ] \text{ }= \sum_{v -\text{ vertex of } t\Delta_{h}} 2\alpha_{2} ( T_{v}(t\Delta_{h})^{\polar})= 2. $$
In the penultimate step in \eqref{Reeve}, we used the Ehrhart–Macdonald reciprocity law $L_{\Delta_{h}}(-t) = (-1)^{3}L_{\relint (\Delta_{h})}(t)$ 
to find the polynomial
\begin{align*}
    L_{\relint (\Delta_{h})} (t) =  \frac{h}{6}t^3 - t^{2} +\left(2 - \frac{h}{6} \right)t -1.
\end{align*}
Thus, 
\begin{align*}
G_{1, \Delta_{h}}(t) = \frac{h}{6}t^3 + t^2 + \left( S - \frac{h}{6} \right) t. 
\end{align*}
The proof is complete.
\end{proof}

\subsection{Theorem~\ref{1206}: Grassmann angle valuations are not combinatorially positive}\label{1207}
\begin{proof}
As noted in the introduction, it suffices to prove that  there is a simplex $\Delta_{0} \in \mathcal{P} (\mathbb{Z}^d) $ such that $G_{k}(\relint(\Delta_{0})):=
     \slim_{F \in \mathcal{F} (\Delta_{0}) }
(-1)^{\dim \Delta_{0} - \dim F} G_{k}(F) < 0 $.
Let $\Delta \in \mathcal{P} (\mathbb{Z}^d)$ be an arbitrary simplex.\\
From the definition of $G_{k}(\relint(\Delta))$ and Theorem \ref{Crofton} it follows that
\begin{align*}
    G_{k}(\relint(\Delta)):&=
     \slim_{F \in \mathcal{F} (\Delta) }
(-1)^{\dim \Delta - \dim F} G_{k}(F)
\\
&=  \slim_{F\in \mathcal{F} (\Delta) }
(-1)^{\dim \Delta - \dim F} \slim_{v \in F\cap\mathbb{Z}^d} \alpha_{k}(T_{v}(F)) 
\\
&= \slim_{F\in \mathcal{F} (\Delta) }
(-1)^{\dim \Delta - \dim F} \slim_{v \in F\cap\mathbb{Z}^d} \sum_{i\geq 1}  \upsilon_{k+i} (T_{v}(F)).
\end{align*}
By Fubini's theorem we have
\begin{align*}
     \slim_{F \in \mathcal{F} (\Delta) }
(-1)^{\dim \Delta - \dim F} &\slim_{v \in F\cap\mathbb{Z}^d} \sum_{i\geq 1}  \upsilon_{k+i} (T_{v}(F))
\\
&= \slim_{v \in \Delta\cap\mathbb{Z}^d}  \sum_{i\geq 1}  \slim_{F \in \mathcal{F} (\Delta) : v\in \mathbb{Z}^d \cap F }
(-1)^{\dim \Delta - \dim F} \upsilon_{k+i} (T_{v}(F)).
\end{align*}
Using \eqref{intrinsic1} and the fact that $T_{v}(F)$ are the faces of the cone $T_{v}(\Delta)$, we rewrite the sums in the following form:
\begin{gather*}
\slim_{F \in \mathcal{F} (\Delta) : v\in \mathbb{Z}^d \cap F }
(-1)^{\dim \Delta- \dim F} \upsilon_{k+i} (T_{v}(F)) 
\\
=(-1)^{\dim \Delta} \slim_{F \in \mathcal{F} (\Delta) : v\in \mathbb{Z}^d \cap F }
(-1)^{\dim F} \upsilon_{k+i} (T_{v}(F))
\\
= (-1)^{\dim \Delta}(-1)^{k+i}  \upsilon_{k+i} (T_{v}(\Delta)).
\end{gather*}
Therefore,
\begin{align*}
G_{k}(\relint(\Delta)):&= 
\slim_{F \in \mathcal{F} (\Delta) }
(-1)^{\dim \Delta - \dim F} G_{k}(F) 
\\
&= (-1)^{\dim \Delta}\slim_{v \in \Delta\cap\mathbb{Z}^d}  \sum_{i\geq 1} (-1)^{k+i}  \upsilon_{k+i} (T_{v}(\Delta)).
\end{align*}
Now, by Crofton's formula (\ref{1818}), 
\begin{gather*}
\sum_{i\geq 1} (-1)^{k+i}  \upsilon_{k+i} (T_{v}(\Delta)) \\
=
\begin{cases}
       \frac{1}{2} (-1)^{k+1} ( \gamma_{k}(T_{v}(\Delta)) - \gamma_{k+1}(T_{v}(\Delta)) )  & \text{if } T_{v}(\Delta) \text{ is not a linear subspace,}
        \\
        (-1)^{\dim \Delta} & \text{if  } T_{v}(\Delta) \text{ is a linear subspace, } k < \dim \Delta,
        \\
        0 & \text{if } T_{v}(\Delta) \text{ is a linear subspace, }  k \geq  \dim \Delta .
        \end{cases}
\end{gather*}
Combining this with the fact that $T_{v}(\Delta)$ is a linear subspace if and only if $v \in \relint(\Delta) $, we get
\begin{gather*}
G_{k}(\relint(\Delta)) = (-1)^{\dim \Delta}\slim_{v \in (\Delta\setminus(\relint(\Delta))\cap\mathbb{Z}^d}  \frac{1}{2} (-1)^{k+1} \left( \gamma_{k}(T_{v}(\Delta)) - \gamma_{k+1}(T_{v}(\Delta)) \right)
\\
+ (-1)^{\dim \Delta}
\begin{cases}
        \slim_{v \in \relint(\Delta)\cap\mathbb{Z}^d}  (-1)^{\dim \Delta}  & \text{if  } T_{v}(\Delta) \text{ is a linear subspace, } k < \dim \Delta
        \\
         \slim_{v \in \relint(\Delta)\cap\mathbb{Z}^d} 0 & \text{if } T_{v}(\Delta) \text{ is a linear subspace, }  k \geq  \dim \Delta 
        \end{cases}
\\
= (-1)^{\dim \Delta +k+1}\frac{1}{2}\slim_{v \in (\Delta\setminus(\relint(\Delta))\cap\mathbb{Z}^d}   ( \gamma_{k}(T_{v}(\Delta)) - \gamma_{k+1}(T_{v}(\Delta)) ) 
\\
+ \begin{cases} 
L(\relint(\Delta)) & \text{if  } T_{v}(\Delta) \text{ is a linear subspace, } k < \dim \Delta
        \\
         0 & \text{if } T_{v}(\Delta) \text{ is a linear subspace, }  k \geq  \dim \Delta.
        \end{cases}
\end{gather*}
Further, as we mentioned in Subsection \ref{Convex cones}, for any convex cone $C \subseteq \R^d $ with $C \neq \{0\}$ we have:
$$1 = \gamma_0(C) \geq  \gamma_1(C) \geq  \ldots \geq   \gamma_d(C) = 0.$$
Since we can find a simplex $\Delta_{0}\in \mathcal{P} (\mathbb{Z}^{d})$ such that $\dim \Delta_0$ and $k \in \{0,\ldots,d-2\}$  have the same parity, $L(\relint(\Delta_0))  = 0$ and $\gamma_{k}(T_{v}(\Delta_0))> \gamma_{k+1}(T_{v}(\Delta_0))$ for some $v\in \Delta_0\setminus\relint(\Delta_{0})$, 
it follows that the $G_{k}(\relint(\Delta_{0}))$ is negative for the simplex $\Delta_{0} $.
Thus, condition \rm(iii) (see Subsection \ref{Positive valuations}) 
does not hold for $\Delta_{0} $, which proves our claim.
\end{proof}
\begin {rem} Let us consider case $k=d-1$. In this case, $G_{d-1}(\relint(\Delta))$ coincides with $ G_{d-1}(\Delta) = A(\Delta)$ 
and hence $G_{d-1}(\cdot) $ is combinatorially positive:
\begin{align*}
G_{d-1}(\relint(\Delta)):&= 
\slim_{F \in \mathcal{F} (\Delta) }
(-1)^{\dim \Delta  -\dim F}G_{d-1}(F)
\\
&= (-1)^{\dim \Delta }\slim_{v \in \Delta\cap\mathbb{Z}^d}  \sum_{i\geq 1} (-1)^{d-1+i}  \upsilon_{d-1+i} (T_{v}(\Delta)) 
\\
&= (-1)^{\dim \Delta  + d}\slim_{v \in \Delta\cap\mathbb{Z}^d}  \upsilon_{d} (T_{v}(\Delta)) 
\\
&= (-1)^{\dim \Delta  + d}\slim_{v \in \Delta\cap\mathbb{Z}^d} \alpha_{d-1} (T_{v}(\Delta))
        =G_{d-1}(\Delta) = A(\Delta).
\end{align*} 
Here the fifth equality follows from the fact that $\alpha_{d-1}(T_{v}(\Delta)) \neq 0 $ if and only if $\dim \Delta = d$.
\end{rem}

\section{Acknowledgments}
The author is grateful to Dmitry Zaporozhets and Anna Gusakova for helpful discussions and valuable remarks.

\bibliographystyle{plain}
\bibliography{diploma}
\vspace{1cm}
\author{Mariia Dospolova
\\
Leonhard Euler International Mathematical Institute\\
Russia\\ 
email \href{mailto:me@somewhere.com}{\color{blue}dospolova.maria@yandex.ru} }
\end{document}